\documentclass[12 pt]{article}%
\usepackage{amsmath, amsfonts, amsthm, color,latexsym}
\usepackage{amsmath}
\usepackage{amsfonts}
\usepackage{amssymb}
\usepackage{color, soul}
\usepackage{enumerate}
\providecommand{\U}[1]{\protect\rule{.1in}{.1in}}
\newtheorem{theorem}{Theorem}[section]
\newtheorem{proposition}[theorem]{Proposition}
\newtheorem{corollary}[theorem]{Corollary}
\newtheorem{example}[theorem]{Example}
\newtheorem{examples}[theorem]{Examples}

\newtheorem{lemma}[theorem]{Lemma}
\newtheorem{final remark}[theorem]{Final Remark}
\newtheorem{definition}[theorem]{Definition}
\allowdisplaybreaks[4]

\begin{document}

\title{\sc Generalized multiple summing multilinear operators on Banach spaces}
\author{Joilson Ribeiro\thanks{joilsonor@ufba.br} ~and Fabr\'icio Santos\thanks{Supported by CAPES doctoral scholarship.	
	\thinspace \hfill\newline\indent2010 Mathematics Subject
Classification: 46B45, 47L22, 40Bxx.\newline\indent Key words: Banach sequence spaces, ideals of multilinear operators, multiple summing operators.}}
\date{}
\maketitle

\begin{abstract} In this paper we provide an abstract aproach to the study of classes of multiple summing multilinear operators between Banach spaces. The main purpose is unify the study of several known classes and results, for example multiple $(p, q_1,\dots, q_n)$-summing operators, multiple mixing $(s, q, p)$-summing operators and multiple strong $(s, q, p)$-mixing summing operators. We also define new class of multiple summing multilinear operator that are particular cases of our construction and, therefore, satisfy the results proved in the paper.
% the  yours homogeneous  polynomials associed, we will also investigate the coerence and compatibility of this class and show that, this is a global holomorphy type. Some inclusion results were also made. Moreover, we show that our approach generalizes several classes of multiple summing operators already established in the literature and that we can easily construct new classes of multiple operators that satisfy our abstract approach.
\end{abstract}

\section{Introduction and background}

%The functional analysis has been started in first decades of the 20th century, seeking to generalise the concepts of convergence and continuity. A important work that evidenced the study of the functional analysis was the work of Banach \cite{Banach32}.

In the 1970's, Pietsch \cite{Pietsch67} introduced the abstract theory of operators ideals and in 1983 he presented in \cite{5} the concept of ideals of multilinear operators, whose adaptation to the case of homogeneous polynomials is immediate.

%In the 1970s, Pietsch \cite{Pietsch67} introduced the abstract theory of operators ideals. In 1983, A. Pietsch in \cite{5} presented the concept of ideals of multilinear operators, whose adaptation to the polynomials $n$-homogeneous is immediate.

The notion of multiple summing multilinear operators was introduced, independently, in \cite{FDI04, 26}. This notion, which is based on the successful theory of absolutely summing operators, has been extensively studied recently. %in thewhich has made significant progress in recent years.
%In 2003, was introduced the notion of multiple summing multilinear operators and polynomials (see \cite{26}).
Several aspects of the theory of summing multilinear operators show
that the class of multiple summing multilinear operators is one of the most suitable and useful approaches to the nonlinear theory of absolutely summing
operators. Details can be found, e.g., in \cite{BPSS15, 19, 20,  38, 21, 34, 39, 26}.

  In \cite{S13} it was presented an abstract approach to absolutely summing operators, that generalizes some concepts of absolutely summing operators already studied in the literature. But the task of generalization is not easy. For example, this work has small gaps that were filled by Botelho and Campos in \cite{BC17}. Such generalizations deal with abstract classes of vector-valued sequences, abstract finitely determined sequence classes and abstract linearly stable sequence classes.
%In \cite{BC17} was introduced the concept abstract class of vector-valued sequences, this concept will be fundamental to our work. By $c_{00}(E)$ we denote the set of all $E$-valued finite sequences, which, as usual, can be regarded as infinite sequences by completing with zeros.
%\cite[Theorem $2$]{S13}
Up to the corrections pointed out in \cite{BC17}, it was proved in \cite{S13} that the abstract classes of summing absolutely multilinear operators are Banach ideals of multilinear operators.

The main goal of this paper is to construct an abstract approach to the classes of  multiple summing multilinear operators and to show that the resulting classes are Banach multi-ideals. Moreover, coherence and compatibility of these multi-ideals will be investigated. In the last section of the paper we will recover some well studied classes as particular instances of our abstract construction and new classes will be introduced as well. % be dedicated to explore several examples that have satisfied ours abstract approach.

Our abstract construction is strongly based on the concept of sequence classes introduced in \cite{BC17}. However, to deal with multiple summing operators we have to extend this notion to what we call $n$-sequences classes %can not use only the concept of sequences, therefore, it was necessary to introduce a more general concept than the sequence classes. This will be approached in
in Section \ref{Sec2}. % of this work and called classes of $n$-sequences vector-valued.
For the moment, let us recall the original definition from  \cite{BC17}:% are introduced in \cite{BC17}.

\begin{definition}
A class of vector-valued sequences $\gamma_s$ is a rule that assigns to each Banach space $E$ a Banach space $\gamma_s(E)$  of $E$-valued sequences, that is
$\gamma_s(E)$ is a vector subspace of $E^{\mathbb{N}}$ with the coordinatewise operations, such that:
\begin{equation*}
c_{00}(E) \subset \gamma_s(E) \overset{1}{\hookrightarrow} \ell_{\infty}(E) \quad \text{and} \quad \|e_j\|_{\gamma_s(\mathbb{K})} = 1, \quad \forall j=1,\dots, n.
\end{equation*}
where $e_j$ is the vector with $1$ in the $j$-th coordinate and zero in the other coordinates, and
the symbol $E \overset{1}{\hookrightarrow} F$ means that $E $ is a linear subspace of $F$ and $\|x\|_F \le \|x\|_E$ .
\end{definition}

 Still according to  \cite{BC17}, the sequence class $\gamma_s$ is said to be:\\
$\bullet$ \textit{finitely determined} if for every $(x_j)_{j=1}^{\infty} \in E^{\mathbb{N}}$, it holds

\begin{equation*}
(x_j)_{j=1}^{\infty} \in \gamma_s(E) \iff \sup_k\left\|\left(x_j \right)_{j=1}^{k} \right\|_{\gamma_s(E)} < \infty,
\end{equation*}
and, in this case,

\begin{equation*}
\left\|\left(x_j \right)_{j=1}^{\infty} \right\|_{\gamma_s(E)} := \sup_k\left\|\left(x_j \right)_{j=1}^{k} \right\|_{\gamma_s(E)}.
\end{equation*}
$\bullet$ {\it linearly stable} if for every $u \in \mathcal{L}(E; F)$ it holds

\begin{equation*}
\left(u\left(x_j \right)\right)_{j=1}^{\infty} \in \gamma_s(F)
\end{equation*}
whenever $\left(x_j \right)_{j=1}^{\infty} \in \gamma_s(E)$ and $\|\hat{u} : \gamma_s(E) \rightarrow \gamma_s(F)\| = \|u\|$, where $\hat u$ is the linear operator induced, in the obvious way, by $u$.
%\end{definition}

%An important result for our work is the following.

Given sequence classes $\gamma_{s_1},\dots, \gamma_{s_n}, \gamma_s$, we write $\gamma_{s_1}(\mathbb{K})\cdots \gamma_{s_n}(\mathbb{K}) \overset{1}{\hookrightarrow} \gamma_s(\mathbb{K})$ if $\left(\lambda_j^{(1)}\cdots \lambda_j^{(n)} \right)_{j=1}^{\infty} \in \gamma_s(\mathbb{K})$ and

\begin{equation*}
\left\|\left(\lambda_j^{(1)}\cdots \lambda_j^{(n)} \right)_{j=1}^{\infty} \right\|_{\gamma_s(\mathbb{K})} \le \prod_{i=1}^{\infty}\left\|\left(\lambda_j^{(i)} \right)_{j=1}^{\infty} \right\|_{\gamma_{s_i}(\mathbb{K})}
\end{equation*}
whenever $\left(\lambda_j^{(i)} \right)_{j=1}^{\infty} \in \gamma_{s_i}(\mathbb{K}), i=1,\dots, n$.

The classical notion of ideal of multilinear operators (multi-ideal) is presented in the next definition.

\begin{definition}
Let $n \in \mathbb{N}$. A Banach ideal of $n$-linear operators is a pair $\left(\mathcal{M}_n, \|\cdot \|_{\mathcal{M}_n} \right)$ where $\mathcal{M}_n$ is as subclass of the class of all $n$-linear operators between Banach spaces and
\begin{equation*}
\|\cdot \|_{\mathcal{M}_n} : \mathcal{M}_n \longrightarrow \mathbb{R}
\end{equation*}
it is a function such that, for all Banach spaces $E_1,\dots, E_n, F$, the component

\begin{equation*}
\mathcal{M}_n(E_1,\dots, E_n; F) := \mathcal{L}(E_1,\dots, E_n; F) \cap \mathcal{M}_n
\end{equation*}
is a subspace of $\mathcal{L}(E_1,\dots, E_n; F)$ on which $\|\cdot \|_{\mathcal{M}_n}$ is a complete norm and:
\begin{enumerate}
\item The subspace of the $n$-linear operators of finite type is contained in $\mathcal{M}_n(E_1,\dots, E_n; F)$.

\item The operator $I_n : \mathbb{K}^n \rightarrow \mathbb{K}$, given by $I_n(\lambda_1,\dots, \lambda_n) = \lambda_1 \cdots \lambda_n$, belongs to $\mathcal{M}_n(\mathbb{K}^n; \mathbb{K})$ and $\|I_n\|_{\mathcal{M}_n} = 1$.

\item $($Multi-ideal property$)$ If $T \in \mathcal{M}_n(E_1,\dots, E_n; F), u_i \in \mathcal{L}(G_i; E_i), i=1,\dots, n$, and $t \in \mathcal{L}(F; H)$ then  $t\circ T \circ (u_1,\dots, u_n) \in \mathcal{M}_n(G_1,\dots, G_n; H)$ and

\begin{equation*}
\|t\circ T \circ (u_1,\dots, u_n) \|_{\mathcal{M}_n} \le \|t\| \left\|T \right\|_{\mathcal{M}_n} \|u_1\|\cdots \|u_n\|.
\end{equation*}
\end{enumerate}
\end{definition}

The notion of ideals of homogeneous polynomials can be defined in a similar way (see, e.g., \cite{13,30}).

%%%%%%%%%%%%%%%%%%%%%%%%%%%%%%%%%%%%%%%%%%%%%%%%%%%%%%%

\section{Multiple $\gamma_{s, s_1,\dots, s_n}$-summing operators}\label{Sec2}

The theory of multiple summing multilinear operators, which has been intensively studied, see, e.g. \cite{BPSS15, 19, 34, 39}, serves as a prototype of the general theory we introduce in this section. We begin presenting the notions of $n$-sequences and classes of vector-valued $n$-sequences.

\begin{definition} Given $n \in \mathbb{N}$, an $n$-sequence in a Banach space $E$ is a map $f : \mathbb{N}^n\rightarrow E$. Writing
\begin{equation*}
f(j_1,\dots, j_n) = x_{j_1,\dots, j_n}~{\rm for~all~} j_1, \ldots, j_n \in \mathbb{N},
\end{equation*}
the $n$-sequence $f$ can be denoted by $(x_{j_1,\dots, j_n})_{j_1,\dots, j_n=1}^{\infty}$.
\end{definition}

It is worth observing that, for $n \geq 2$, an $n$-sequence is not a sequence. For example, given a $2$-sequence $(x_{i, j})_{i, j = 1}^{\infty}$, it is useless to try to display it like a sequence in the following fashion:
\begin{equation*}
(x_{1, 1}, x_{2, 1}, x_{3, 1},\dots, x_{1, 2}, x_{2, 2}, x_{3, 2},\dots, x_{1, 3}, x_{2, 3}, x_{3, 3},\dots).
\end{equation*}
Of course this is not a sequence. Note also that an $n$-sequence can be transformed into a sequence in several ways. For example, consider the $2$-sequence $(x_{i, j})_{i, j \in \mathbb{N}}$ given by
\begin{equation*}
x_{i, j} = \left\{\begin{array}{rc}
1, &\text{if $i$ is odd and any $j \in \mathbb{N}$} \\
0, &\text{if $i$ is pair and any $j \in \mathbb{N}$}.
\end{array}\right.
\end{equation*}
It can be transformed into a sequence in many ways, for example
 \begin{align*}
 (1, 0, 1, 0,\dots) \text{\ \ \ \  and \ \ \ \  }  (0, 1, 0, 1,\dots).
 \end{align*}
% and
% \begin{align*}
% (0, 1, 0, 1,\dots).
% \end{align*}
% We can also take the $2$ -sequence, such that there are elements whose position is infinite, for example, consider the $2$-sequ\^encia $(x_{i, j})_{i, j \in \mathbb{N}}$. We can enumerate of the following way
%\begin{equation*}
%(x_{1, 1}, x_{2, 1}, x_{3, 1},\dots, x_{1, 2}, x_{2, 2}, x_{3, 2},\dots, x_{1, 3}, x_{2, 3}, x_{3, 3},\dots).
%\end{equation*}
%This way, the element $x_{1, 2}$ have infinite position in the $2$-sequence. What does not happen in the case of sequences.

Now we start the construction of our abstract framework. Throughout this paper, we will consider:
\begin{align*}
c_{00}\left(E; \mathbb{N}^n \right)
:= \left\{(x_{j_1,\dots, j_n})_{j_1,\dots, j_n =1}^{\infty}\text{; } x_{j_1,\dots, j_n} \neq 0 \text{ for only finitely many } j_1,\dots, j_n \right\}
\end{align*}
and
\begin{align*}
\ell_{\infty}\left(E; \mathbb{N}^n\right) := \left\{(x_{j_1,\dots, j_n})_{j_1,\dots, j_n =1}^{\infty} \text{; } \sup_{j_1,\dots, j_n \in \mathbb{N}}\|x_{j_1,\dots, j_n}\| < \infty \right\}.
\end{align*}

\begin{definition}{}
Let $n \in \mathbb{N}$. A class of vector-valued $n$-sequences $\gamma_s\left(\cdot, \mathbb{N}^n \right)$, or simply an $n$-sequence class $\gamma_s\left(\cdot, \mathbb{N}^n \right)$, is a rule that assigns to each Banach space $E$ a Banach space $\gamma_s(E; \mathbb{N}^n)$  of $E$-valued $n$-sequences, that is $\gamma_s(E; \mathbb{N}^n)$ is a complete linear subspace of the space of all $E$-valued $n$-sequences with the coordinatewise operations, such that:
\begin{equation*}
c_{00}\left(E; \mathbb{N}^n\right) \subset \gamma_s\left(E; \mathbb{N}^n\right) \overset{1}{\hookrightarrow} \ell_{\infty}\left(E; \mathbb{N}^n\right) \quad \text{and} \quad \|e_{k_1,\dots, k_n}\|_{\gamma_s\left(\mathbb{K}; \mathbb{N}^n\right)} = 1,
\end{equation*}
for all $k_1,\dots, k_n \in \mathbb{N}$, where $e_{k_1,\dots, k_n}=\left(x_{j_1,\dots,j_n}\right)_{j_1,\dots,j_n=1}^{\infty}$ is the $n$-sequence defined by:
	\begin{equation*}
	x_{j_1,\dots,j_n} = \left\{\begin{array}{rc}
	1, &\text{if} \quad j_1=k_1,\dots, j_n=k_n\\
	0, &\text{otherwise}.
	\end{array}\right.
	\end{equation*}
%and the symbol $E \overset{1}{\hookrightarrow} F$ means that $E $ is a linear subspace of $F$ and $\|x\|_F \le \|x\|_E$.
\end{definition}

An $n$-sequence class $\gamma_s\left(\cdot; \mathbb{N}^n \right)$ is \textit{finitely determined} if for every $E$-valued $n$-sequence $(x_{j_1,\dots, j_n})_{j_1,\dots, j_n = 1}^{\infty}$,
\begin{equation*}
(x_{j_1,\dots, j_n})_{j_1,\dots, j_n = 1}^{\infty} \in \gamma_s\left(E; \mathbb{N}^n \right) \iff \sup_{m_1,\dots, m_n\in \mathbb{N}}\left\|\left(x_{j_1,\dots, j_n} \right)_{j_1,\dots, j_n=1}^{m_1,\dots, m_n} \right\|_{\gamma_s\left(E; \mathbb{N}^n \right)}<\infty
\end{equation*}
and, in this case,
\begin{equation*}
\left\|(x_{j_1,\dots, j_n})_{j_1,\dots, j_n = 1}^{\infty} \right\|_{\gamma_s\left(E; \mathbb{N}^n \right)} = \sup_{m_1,\dots, m_n \in \mathbb{N}}\left\|\left(x_{j_1,\dots, j_n} \right)_{j_1,\dots, j_n=1}^{m_1,\dots, m_n} \right\|_{\gamma_s\left(E; \mathbb{N}^n \right)}.
\end{equation*}

Note that the concept of $n$-sequence class generalizes the concept of sequence class introduced in the literature by Botelho and Campos in \cite{BC17}.

Next we give some examples of finitely determined $n$-sequence classes.
\begin{examples}\label{ExSc}\
\begin{description}
    \item $(a)$ The correspondence $E \mapsto \ell_{\infty}\left(E; \mathbb{N}^n \right)$, endowed with the norm
        $$\|(x_{j_1,\dots, j_n})_{j_1,\dots, j_n = 1}^{\infty}\|_{\infty} := \displaystyle\sup_{j_1,\dots, j_n \in \mathbb{N}}\|x_{j_1,\dots, j_n}\|.$$
    \item $(b)$ The correspondence $E \mapsto \ell_p^w\left(E; \mathbb{N}^n \right)$, $1 \leq p < \infty$, where
    \begin{equation*}
	\ell_p^w\left(E; \mathbb{N}^n \right) := \left\{(x_{j_1,\dots, j_n})_{j_1,\dots, j_n = 1}^{\infty}\text{; }\sum_{j_1,\dots,j_n}|\varphi\left(x_{j_1,\dots, j_n}\right)|^p < \infty\text{, } \forall \varphi \in E' \right\},
	\end{equation*}
	endowed with the norm $\|(x_{j_1,\dots, j_n})_{j_1,\dots, j_n = 1}^{\infty}\|_{w, p} := \displaystyle\sup_{\varphi \in B_{E'}} \left(\displaystyle\sum_{j_1,\dots,j_n}|\varphi\left(x_{j_1,\dots, j_n}\right)|^p \right)^{\frac{1}{p}}$.
	\item $(c)$ The correspondence $E \mapsto \ell_p\left(E; \mathbb{N}^n \right)$, $1 \leq p < \infty$, where
	\begin{equation*}
	\ell_p\left(E; \mathbb{N}^n \right) := \left\{(x_{j_1,\dots, j_n})_{j_1,\dots, j_n = 1}^{\infty}\text{; }\sum_{j_1,\dots,j_n}\|x_{j_1,\dots, j_n}\|^p < \infty \right\},
	\end{equation*}
endowed with the norm $\|(x_{j_1,\dots, j_n})_{j_1,\dots, j_n = 1}^{\infty}\|_{p} := \left(\displaystyle\sum_{j_1,\dots,j_n}\|x_{j_1,\dots, j_n}\|^p \right)^{\frac{1}{p}}$.
	\item $(d)$ The correspondence $E \mapsto \ell_p\left<E; \mathbb{N}^n \right>$, $1 \le p < \infty$, where
	\begin{align*}
	&\ell_p\left<E; \mathbb{N}^n \right>\\
	:=& \left\{(x_{j_1,\dots, j_n})_{j_1,\dots, j_n = 1}^{\infty}\text{;}\sum_{j_1,\dots, j_n}\left|\varphi_{j_1,\dots, j_n}\left(x_{j_1,\dots, j_n}\right)\right| < \infty,\text{ } \forall (\varphi_{j_1,\dots, j_n})_{j_1,\dots, j_n \in \mathbb{N}} \in \ell_{p'}^w\left(E'; \mathbb{N}^n \right) \right\},
	\end{align*}
endowed with the norm $$\|(x_{j_1,\dots, j_n})_{j_1,\dots, j_n = 1}^{\infty}\|_{Coh, p} := \displaystyle\sup_{(\varphi_{j_1,\dots, j_n})_{j_1,\dots, j_n=1}^{\infty} \in B_{\ell_p^w\left(E'; \mathbb{N}^n\right)}}\displaystyle\sum_{j_1,\dots, j_n}\left|\varphi_{j_1,\dots, j_n}\left(x_{j_1,\dots, j_n}\right)\right|.$$
	\item $(e)$ The correspondence $E \mapsto \ell_p^{mid}\left(E; \mathbb{N}^n \right)$, $1 \leq p < \infty$, where
	\begin{align*}
	&\ell_p^{mid}\left(E; \mathbb{N}^n \right)
	:= \left\{(x_{j_1,\dots, j_n})_{j_1,\dots, j_n = 1}^{\infty}\text{; }\sum_{n}\sum_{j_1,\dots, j_n}|\varphi_n(x_{j_1,\dots, j_n})|^p < \infty, \text{ } \forall (\varphi_n)_{n=1}^{\infty} \in \ell_p^w(E') \right\},
	\end{align*}
endowed with the norm $\|(x_{j_1,\dots, j_n})_{j_1,\dots, j_n = 1}^{\infty}\|_{mid, p} := \displaystyle\sup_{(\varphi_n)_{n=1}^{\infty} \in B_{\ell_p^w(E')}}\left(\displaystyle\sum_{n}\displaystyle\sum_{j_1,\dots, j_n}|\varphi_n(x_{j_1,\dots, j_n})|^p\right)^{\frac{1}{p}}$.
    \item $(f)$ The correspondence $E \mapsto \ell_{m(s, p)}\left(E; \mathbb{N}^n \right)$, $1 \le p < \infty$, where $\ell_{m(s, p)}\left(E; \mathbb{N}^n \right)$ is the set of all $E$-valued $n$-sequences $(x_{j_1,\dots, j_n})_{j_1,\dots, j_n=1}^{\infty}$ such that $x_{j_1,\dots, j_n} = \tau_{j_1,\dots, j_n}x_{j_1,\dots, j_n}^0$ for some $(\tau_{j_1,\dots, j_n})_{j_1,\dots, j_n=1}^{\infty} \in  \ell_{s(q)'}\left(\mathbb{K}; \mathbb{N}^n \right), (x_{j_1,\dots, j_n}^0)_{j_1,\dots, j_n=1}^{\infty} \in \ell_p^w\left(E; \mathbb{N}^n \right)$ and $$\frac{1}{s(p)'} + \frac{1}{s} = \frac{1}{p}.$$
    Consider $\ell_{m(s, q)}\left(E; \mathbb{N}^n \right)$ endowed with the norm
\begin{equation*}
\left\|\left(x_{j_1,\dots, j_n} \right)_{j_1,\dots, j_n=1}^{\infty} \right\|_{m(s; q)} = \inf \left\|\left(\tau_{j_1,\dots, j_n} \right)_{j_1,\dots, j_n=1}^{\infty} \right\|_{s(q)'}\left\|\left(x_{j_1,\dots, j_n}^0 \right)_{j_1,\dots, j_n=1}^{\infty} \right\|_{w, s},
\end{equation*}
where the infimum ranges over all representations $x_{j_1,\dots, j_n} = \tau_{j_1,\dots, j_n} x_{j_1,\dots, j_n}^0$, $j_1,\dots, j_n \in \mathbb{N}$.
\end{description}
\end{examples}
%For simplicity we will to represent the $n$-sequence class $\gamma_s\left(\cdot; \mathbb{N}^n \right)$ just by $\gamma_s$.

% satisfying the condition above.

Note that considering $(x_{j_1,\dots, j_n})_{j_1,\dots, j_n  = 1}^{\infty}$ in different orders of the indexes, one may end up with different $n$-sequences. In order to avoid this dependence on the order the indexes are taken, we shall henceforth consider only $n$-sequence classes $\gamma_s\left(\cdot, \mathbb{N}^n \right)$ enjoying the following symmetry condition: for any $E$-valued $n$-sequence $(x_{j_1,\dots, j_n})_{j_1,\dots, j_n  = 1}^{\infty}$,
%Given a Banach space $E$, we say that an $n$-sequence class $\gamma_s\left(\cdot, \mathbb{N}^n \right)$ is symmetrically closed if
\begin{equation*}
(x_{j_1,\dots, j_n})_{j_1,\dots, j_n  = 1}^{\infty} \in \gamma_s\left(E; \mathbb{N}^n \right) \iff (x_{j_{\sigma(1)},\dots,j_{\sigma(n)}})_{j_1,\dots, j_n = 1}^{\infty} \in \gamma_s\left(E; \mathbb{N}^n \right)
\end{equation*}
and%, in this case,
$$\left\|(x_{j_1,\dots, j_n})_{j_1,\dots, j_n = 1}^{\infty} \right\|_{\gamma_s\left(E; \mathbb{N}^n \right)} = \left\|(x_{j_{\sigma(1)},\dots, j_{\sigma(n)}})_{j_1,\dots, j_n = 1}^{\infty} \right\|_{\gamma_s\left(E; \mathbb{N}^n \right)},$$
for every permutation $\sigma$ of the set $\{1, \ldots,n\}$.
%\end{definition}

%Note that, when we change the order of the index variation, we can obtain distinct $n$-sequences. However, if the $n$-sequence class $\gamma_s\left(\cdot; \mathbb{N}^n\right)$ is symmetrically closed, it ensures that the $n$-sequences obtained by index variation will still belong to the same $n$-sequences space $\gamma_s\left(E; \mathbb{N}^n\right)$ and have the same norm as the original $n$-sequence. This property will be of great importance in defining the multiple $\gamma_{s, s_1,\dots, s_n}$-summing operators, which will be introduced in the next definition.
%
%From here, all classes of n-sequences considered will be symmetrically closed.
Henceforth $\gamma_{s_1},\dots, \gamma_{s_n}$ are linearly stable finitely determined sequence classes and $\gamma_s\left(\cdot, \mathbb{N}^n \right)$ is a finitely determined $n$-sequence class enjoying the symmetry condition above.

Now we are ready to introduce our main definition:

\begin{definition}\label{MSO}
Let $n \in \mathbb{N}$. A continuous $n$-linear operator $T\in\mathcal{L}(E_1,\dots,E_n;F)$ is said to be multiple $\gamma_{s, s_1,\dots, s_n}$-summing if
\begin{equation*}
 \left(T\left(x_{j_1}^{(1)},\dots, x_{j_n}^{(n)}\right)\right)_{j_1,\dots, j_n=1}^{\infty} \in \gamma_s\left(F; \mathbb{N}^n\right)
\end{equation*}
whenever $\left(x_{j}^{(i)} \right)_{j=1}^{\infty} \in \gamma_{s_i}(E_i), i=1,\dots, n$.
\end{definition}

Note that the symmetry condition of the $n$-sequence class $\gamma_s\left(\cdot, \mathbb{N}^n \right)$ guarantees that this concept is well defined. % class Observe that the above property ensures the good definition of the multiple $\gamma_{s, s_1,\dots, s_n}$-summing operators. Indeed, denoting $\left(T\left(x_{j_1}^{(1)},\dots, x_{j_n}^{(n)}\right)\right)_{j_1,\dots, j_n=1}^{\infty} = \left(y_{j_1,\dots, j_n}\right)_{j_1,\dots, j_n=1}^{\infty}$, follows directly of property that $$\left(y_{{j_{\sigma(1)}},\dots,{j_{\sigma(n)}}}\right)_{j_1,\dots, j_n=1}^{\infty} = \left(T\left(x_{j_{\sigma(1)}}^{(1)},\dots, x_{j_{\sigma(n)}}^{(n)}\right)\right)_{j_1,\dots, j_n=1}^{\infty} \in \gamma_s\left(F; \mathbb{N}^n\right).$$

The set of all multiple $\gamma_{s, s_1,\dots, s_n}$-summing $n$-linear operators from $E_1 \times \cdots \times E_n$ to $F$, which is clearly a linear space, is denoted by $\mathcal{L}_{\gamma_{s, s_1,\dots, s_n}}^m(E_1,\dots, E_n; F)$. To give this space a suitable complete norm, we need the following result.%Clearly $\mathcal{L}_{\gamma_{s, s_1,\dots, s_n}}^m(E_1,\dots, E_n; F)$ is a linear space.

%A important result which is immediately  from the Definition \ref{MSO}, is the following:

%The next proposition will be important because, it will to introduce the norm that will make the linear space $\mathcal{L}_{\gamma_{s, s_1,\dots, s_n}}^m(E_1,\dots, E_n; F)$ a Banach space. We will start by showing the following lemma.
%Let's show that the infimum of the constant $C>0$ satisfyng \eqref{EEEE1.2} is a norm on $\mathcal{L}_{\gamma_{s, s_1,\dots, s_n}}(E_1,\dots, E_n; F)$, that will be denoted by $\|\cdot \|_{\mathcal{L}_{\gamma_{s, s_1,\dots, s_n}}}$.

\begin{lemma}\label{LLL2.2.}
Let $E_1,\dots, E_n$, $F$ be Banach spaces and $T \in \mathcal{L}_{\gamma_{s, s_1,\dots, s_n}}^m(E_1,\dots, E_n; F) $. Then, the induced map
\begin{equation*}
\hat{T} : \gamma_{s_1}(E_1) \times \cdots \times \gamma_{s_n} (E_n) \rightarrow \gamma_s\left(F; \mathbb{N}^n \right)
\end{equation*}
given by
\begin{equation*}
\hat{T}\left(\left(x_{j_{1}}^{(1)} \right)_{j_{1}=1}^{\infty},\dots, \left(x_{j_{n}}^{(n)}\right)_{j_{n}=1}^{\infty} \right) = \left(T\left(x_{j_1}^{(1)},\dots, x_{j_n}^{(n)}\right) \right)_{j_1,\dots, j_n=1}^{\infty}
\end{equation*}
is well-defined continuous $n$-linear operator.
\end{lemma}

\begin{proof}
It is clear that $\hat{T}$ is well-defined and it is easy to check its $n$-linearity. To prove that $\hat{T} $ is a continuous operator, we will use the Closed Graph Theorem. We shall use the sum norm in the cartesian product.
%Let~s show that $\hat{T}$ is continuous, for this we will use the Closed Graph Theorem.
%Segue diretamente da defini\c c\~ao que $\hat{T}$ est\'a bem definido e que \'e multilinear. Vamos mostrar que $\hat{T}$ \'e cont\'inua, para isso utilizaremos o Teorema do Gr\'afico Fechado.
Let
\begin{equation*}
\left(\left(x_{j_1}^{(1), k} \right)_{j_1=1}^{\infty},\dots, \left(x_{j_n}^{(n), k} \right)_{j_n=1}^{\infty} \right)_{k=1}^{\infty}
\end{equation*}
be a sequence in $\gamma_{s_1}(E_1) \times \cdots \times \gamma_{s_n} (E_n)$ converging to
\begin{equation*}
\left(\left(x_{j_1}^{(1)} \right)_{j_1=1}^{\infty},\dots, \left(x_{j_n}^{(n)} \right)_{j_n=1}^{\infty}\right)
\end{equation*}
and such that
\begin{equation}\label{EEEEE2}
\hat{T}\left(\left(x_{j_1}^{(1), k} \right)_{j_1=1}^{\infty},\dots, \left(x_{j_n}^{(n), k} \right)_{j_n=1}^{\infty}\right) \underset{k \rightarrow \infty}{\longrightarrow}  \left(z_{j_1,\dots, j_n}\right)_{j_1,\dots, j_n=1}^{\infty}.
\end{equation}
So, given $\epsilon > 0$ there exists $k_{0} \in \mathbb{N}$ such that for any  $k \ge k_0$ and $i=1,\dots, n$,
%\begin{equation*}
%\left\|\left(\left(x_{j_1}^{(1), k} \right)_{j_1=1}^{\infty},\dots, \left(x_{j_n}^{(n), k} \right)_{j_n=1}^{\infty} \right) - \left(\left(x_{j_1}^{(1)} \right)_{j_1=1}^{\infty},\dots, \left(x_{j_n}^{(n)} \right)_{j_n=1}^{\infty}\right) \right\|_{\gamma_{s_1}(E_1) \times \cdots \times \gamma_{s_n} (E_n)} < \epsilon.
%\end{equation*}
\begin{align*}
&\left\|x_{j_i}^{(i), k} - x_{j_i}^{(i)} \right\|_{E_i}\le \left\|\left(x_{j_i}^{(i), k} \right)_{j_i=1}^{\infty} - \left(x_{j_i}^{(i)} \right)_{j_i=1}^{\infty} \right\|_{\gamma_{s_i}(E_i)} \\
&\le \left\|\left(\left(x_{j_1}^{(1), k} \right)_{j_1=1}^{\infty},\dots, \left(x_{j_n}^{(n), k} \right)_{j_n=1}^{\infty} \right) - \left(\left(x_{j_1}^{(1)} \right)_{j_1=1}^{\infty},\dots, \left(x_{j_n}^{(n)} \right)_{j_n=1}^{\infty}\right) \right\|_{\gamma_{s_1}(E_1) \times \cdots \times \gamma_{s_n} (E_n)}< \epsilon.
\end{align*}
Thus, $x_{j_i}^{(i), k} \rightarrow x_{j_i}^{(i)}$ for every $i=1,\dots, n$. It follows from the continuity of $T$ that
\begin{equation*}
T\left(x_{j_1}^{(1), k},\dots, x_{j_n}^{(n), k} \right) \underset{k \rightarrow \infty}{\longrightarrow} T\left(x_{j_1}^{(1)} ,\dots, x_{j_n}^{(n)}  \right),
\end{equation*}
for all $j_1,\dots, j_n \in \mathbb{N}$. On the other hand, by \eqref{EEEEE2} we can take $k_1 \in \mathbb{N}$ such that for any $k \ge k_1$ and $i=1,\dots, n$,
%
%\begin{equation*}
%\left\|\hat{T}\left(\left(x_{j_{1}}^{(1),k} \right)_{j_{1}=1}^{\infty},\dots, \left(x_{j_{n}}^{(n),k}\right)_{j_{n}=1}^{\infty} \right) - \left(z_{j_1,\dots, j_n}\right)_{j_1,\dots, j_n=1}^{\infty} \right\|_{\gamma\left(F; \mathbb{N}^n\right)} < \epsilon.
%\end{equation*}
%
\begin{align*}
&\left\|T\left(x_{j_1}^{(1), k},\dots, x_{j_n}^{(n), k} \right) - z_{j_1,\dots, j_n} \right\|_F \le \left\| \left( T\left(x_{j_1}^{(1), k},\dots, x_{j_n}^{(n), k}\right) \right)_{j_1,\dots, j_n=1}^{\infty} - \left(z_{j_1,\dots, j_n}\right)_{j_1,\dots, j_n=1}^{\infty}  \right\|_{\gamma\left(F; \mathbb{N}^n\right)}\\
&= \left\|\hat{T}\left(\left(x_{j_{1}}^{(1),k} \right)_{j_{1}=1}^{\infty},\dots, \left(x_{j_{n}}^{(n),k}\right)_{j_{n}=1}^{\infty} \right) - \left(z_{j_1,\dots, j_n}\right)_{j_1,\dots, j_n=1}^{\infty} \right\|_{\gamma\left(F; \mathbb{N}^n\right)}< \epsilon.
\end{align*}
So,
\begin{align*}
T\left(x_{j_1}^{(1), k},\dots, x_{j_n}^{(n), k} \right) \underset{k \rightarrow \infty}{\longrightarrow} z_{j_1,\dots, j_n}
\end{align*}
for all $ j_1,\dots, j_n \in \mathbb{N}$. Consequently,
%\begin{align*}
%T\left(x_{j_1}^{(1)} ,\dots, x_{j_n}^{(n)}  \right) =  z_{j_1,\dots, j_n},
%\end{align*}
%for every $ j_1,\dots, j_n \in \mathbb{N}$. Thus,
\begin{align*}
\hat{T}\left(\left(x_{j_{1}}^{(1)} \right)_{j_{1}=1}^{\infty},\dots, \left(x_{j_{n}}^{(n)}\right)_{j_{n}=1}^{\infty} \right) &=  \left( T\left(x_{j_1}^{(1)},\dots, x_{j_n}^{(n)}\right) \right)_{j_1,\dots, j_n=1}^{\infty}=  \left(z_{j_1,\dots, j_n}\right)_{j_1,\dots, j_n=1}^{\infty} ,
\end{align*}
proving that $\hat{T}$ has closed graph, hence it is continuous.
\end{proof}

The converse of Lemma \ref{LLL2.2.} is obviously true. The next proposition will be useful to introduce the norm that will make the linear space $\mathcal{L}_{\gamma_{s, s_1,\dots, s_n}}^m(E_1,\dots, E_n; F)$ a Banach space.

\begin{theorem}\label{PPP2.2.}
Let $E_1,\dots, E_n, F$ be Banach spaces and $T \in \mathcal{L}(E_1,\dots, E_n; F)$. The following statements are equivalent:

\begin{description}
\item $($a$)$ $T \in \mathcal{L}_{\gamma_{s, s_1,\dots, s_n}}^m(E_1,\dots, E_n; F)$.
\item $($b$)$ There is $C > 0$, such that
\begin{equation}\label{EEEE1.2}
\left\|\left(T\left(x_{j_1}^{(1)},\dots, x_{j_n}^{(n)}\right) \right)_{j_1,\dots, j_n = 1}^{\infty}\right\|_{\gamma_s\left(F; \mathbb{N}^n\right)} \le C \prod_{i=1}^n\left\|\left(x_{j}^{(i)} \right)_{j=1}^{\infty} \right\|_{\gamma_{s_i}(E_i)},
\end{equation}
whenever $\left(x_j^{(i)} \right)_{j=1}^{\infty} \in \gamma_{s_i}(E_i), i=1,\dots, n$.

\item $($c$)$ There is $C > 0$ such that
\begin{equation*}
\left\|\left(T\left(x_{j_1}^{(1)},\dots, x_{j_n}^{(n)}\right) \right)_{j_1,\dots, j_n=1}^{m_1,\dots, m_n}\right\|_{\gamma_s\left(F; \mathbb{N}^n\right)} \le C \prod_{i=1}^n\left\|\left(x_{j}^{(i)} \right)_{j=1}^{m_i} \right\|_{\gamma_{s_i}(E_i)},
\end{equation*}
for any $m_1,\dots, m_n \in \mathbb{N}$ and $x_j^{(i)} \in E_i$, $i=1,\dots, n$.
\end{description}
%$T \in \mathcal{L}_{\gamma_{s, s_1,\dots, s_n}}(E_1,\dots, E_n; F)$ if, and only if

\end{theorem}

\begin{proof}
$(\text{b}) \Rightarrow (\text{c})$ is straightforward and $(\text{c}) \Rightarrow (\text{a})$ follows easily from the fact that the underlying sequence classes and  $n$-sequence class are finitely determined. Thus, we have only to prove
$(\text{a}) \Rightarrow (\text{b})$. To do so, suppose that $T \in \mathcal{L}_{\gamma_{s, s_1,\dots, s_n}}^m(E_1,\dots, E_n; F)$ and define

\begin{equation*}
\hat{T} : \gamma_{s_1}(E_1)\times \cdots \times \gamma_{s_n}(E_n) \rightarrow \gamma_s\left(F; \mathbb{N}^n\right),
\end{equation*}
as in Lemma
%\begin{equation*}
%\hat{T}\left(\left(x_{j_1}^{(1)} \right)_{j_1=1}^{\infty},\dots, \left(x_{j_n}^{(n)} \right)_{j_n=1}^{\infty} \right) =  \left(T\left(x_{j_1}^{(1)},\dots, x_{j_n}^{(n)} \right) \right)_{j_1,\dots, j_n=1}^{\infty},
%\end{equation*}
%as done in the Lemma
\ref{LLL2.2.}. Thus, $\hat T$ is a well-defined $n$-linear continuous operator, from which it follows that
\begin{align*}
\left\|\left(T\left(x_{j_1}^{(1)},\dots, x_{j_n}^{(n)} \right) \right)_{j_1,\dots, j_n=1}^{\infty} \right\|_{\gamma_s\left(F; \mathbb{N}^n\right)} &= \left\|\hat{T}\left(\left(x_{j_1}^{(1)} \right)_{j_1},\dots, \left(x_{j_n}^{(n)} \right)_{j_n} \right) \right\|_{\gamma_s\left(F; \mathbb{N}^n\right)}\\
&\le \left\|\hat{T}\right\|\prod_{i=1}^n\left\|\left(x_{j_i}^{(i)} \right)_{j_i = 1}^{\infty} \right\|_{\gamma_{s_i}(E_i)}.
\end{align*}
\end{proof}

Now standard arguments give the following result:%

%is an immediate consequence of the above theorem.

\begin{proposition}\label{PPP3.3.}
Let $E_1,\dots, E_n$ and $F$ be Banach spaces. \\
{\rm (a)} The infimum of the constant $C > 0$ satisfying \eqref{EEEE1.2} in Proposition \ref{PPP2.2.} defines a norm on $\mathcal{L}_{\gamma_{s, s_1,\dots, s_n}}^m(E_1,\dots, E_n; F)$, which is denoted by $\|\cdot \|_{\mathcal{L}_{\gamma_{s, s_1,\dots, s_n}}^m}$.\\
{\rm (b)} $\|T\| \le  \| T\|_{\mathcal{L}_{\gamma_{s, s_1,\dots, s_n}}^m}$ for every $T \in \mathcal{L}_{\gamma_{s, s_1,\dots, s_n}}^m(E_1,\dots, E_n; F)$.\\
{\rm (c)} $\left(\mathcal{L}_{\gamma_{s, s_1,\dots, s_n}}^m(E_1,\dots, E_n; F), \| \cdot\|_{\mathcal{L}_{\gamma_{s, s_1,\dots, s_n}}^m} \right)$ is a Banach space.
\end{proposition}

%The demonstration of this corollary utilize standard arguments and, for that, we will not presented it here.

%We will not presented the demonstration this corollary, because, the argument utilized in this demonstration be standard argument.

%A result that will be important for next definition is the following.

%Um resultado que ser\'a importante para a pr\'oxima defini\c c\~ao \'e o seguinte.

% In this work, we will to consider that $\gamma_s\left(\cdot; \mathbb{N}^n \right)$ is linearly stable $n$-sequence class.

%\begin{remark}\label{R2.9.}
% Let $\gamma_s\left(\cdot; \mathbb{N}^n \right)$ be a linearly stable, finitely determined $n$-sequence class, for every $n \in \mathbb{N}$. Then, $\gamma_s(\cdot) := \gamma_s\left(\cdot; \mathbb{N}^1 \right)$ is a linearly stable, finitely determined  sequence class.
%\end{remark}

\section{$\left(\mathcal{L}_{\gamma_{s, s_1,\dots, s_n}}^m, \| \cdot\|_{\mathcal{L}_{\gamma_{s, s_1,\dots, s_n}}^m} \right)$ is a Banach ideal}\label{Sec3}

After proving that % We will start by presenting some auxiliary results that will be very useful for the main goal of this section, which is to show that the
$\left(\mathcal{L}_{\gamma_{s, s_1,\dots, s_n}}^m, \| \cdot\|_{\mathcal{L}_{\gamma_{s, s_1,\dots, s_n}}^m} \right)$ is a Banach ideal of multilinear operators, we will establish the relationship between the class of multiple $\gamma_{s, s_1,\dots, s_n}$-summing operators and the class of $\gamma_{s, s_1,\dots, s_n}$-summing operators at the origin, introduced in \cite{S13}.

The proof that $\left(\mathcal{L}_{\gamma_{s, s_1,\dots, s_n}}^m, \| \cdot\|_{\mathcal{L}_{\gamma_{s, s_1,\dots, s_n}}^m} \right)$ is a Banach ideal shall be splitted into several steps. The first step was taken in Proposition \ref{PPP3.3.}. Before proceeding to the next steps, let us define the conditions the sequences classes shall be supposed to satisfy.

\begin{definition}
%Let $\gamma_s\left(\cdot, \mathbb{N}^n \right)$ an $n$-sequence class.
We say that an $n$-sequence class $\gamma_s\left(\cdot, \mathbb{N}^n \right)$ is linearly stable if \\$\left(u(x_{j_1,\dots, j_n}) \right)_{j_1,\dots, j_n = 1}^{\infty} \in \gamma_s\left(F; \mathbb{N}^n \right)$ whenever $\left(x_{j_1,\dots, j_n} \right)_{j_1,\dots, j_n = 1}^{\infty} \in \gamma_s\left(E; \mathbb{N}^n \right)$ and $u \in \mathcal{L}(E; F)$, and, in this case,
	\begin{equation*}
	\left\|\hat{u} : \gamma_s\left(E; \mathbb{N}^n\right) \rightarrow \gamma_s\left(F; \mathbb{N}^n \right) \right\| = \|u\|.
	\end{equation*}
\end{definition}
 All $n$-sequence classes presented in Example \ref{ExSc} are linearly stable. From now on, all $n$-sequence classes are supposed to be finitely determined and linearly stable.

Given sequence classes $\gamma_{s_1},\dots, \gamma_{s_n}$ and an $n$-sequence class $\gamma_s\left(\cdot; \mathbb{N}^n \right)$, we write
$$\gamma_{s_1}(\mathbb{K})\cdots \gamma_{s_n}(\mathbb{K}) \overset{mult, 1}{\hookrightarrow} \gamma_s\left(\mathbb{K}; \mathbb{N}^n\right)$$
if $\left(\lambda_{j_1}^{(1)}\cdots \lambda_{j_n}^{(n)} \right)_{j_1,\dots, j_n=1}^{\infty} \in \gamma_s\left(\mathbb{K}; \mathbb{N}^n\right)$ and
\begin{equation*}
\left\|\left(\lambda_{j_1}^{(1)}\cdots \lambda_{j_n}^{(n)} \right)_{j_1,\dots, j_n=1}^{\infty} \right\|_{\gamma_s\left(\mathbb{K}; \mathbb{N}^n\right)} \le \prod_{i=1}^{\infty}\left\|\left(\lambda_j^{(i)} \right)_{j=1}^{\infty} \right\|_{\gamma_{s_i}(\mathbb{K})},
\end{equation*}
whenever $\left(\lambda_j^{(i)} \right)_{j=1}^{\infty} \in \gamma_{s_i}(\mathbb{K}), i=1,\dots, n$.

\begin{proposition}\label{PPP3.4.} Suppose that $\gamma_{s_1}(\mathbb{K})\cdots \gamma_{s_n}\left(\mathbb{K}\right) \overset{mult, 1}{\hookrightarrow} \gamma_s\left(\mathbb{K}; \mathbb{N}^n\right)$. Then, for all Banach spaces $E_1,\dots, E_n$ and $F$, the finite type $n$-linear continuous operators are contained in $\mathcal{L}_{\gamma_{s, s_1,\dots, s_n}}^m(E_1,\dots, E_n; F)$.
\end{proposition}

\begin{proof}
Consider the $n$-linear operator
\begin{equation*}
B : E_1\times \cdots \times E_n \longrightarrow F~,~
B(x_1,\dots, x_n) = \varphi^{(1)}(x_1)\cdots \varphi^{(n)}(x_n)b,
\end{equation*}
where $b \in F$, $\varphi^{(r)} \in E_r'$, $r=1,\dots, n$. Consider also the linear operator
$$
f : \mathbb{K} \rightarrow F~,~f(\lambda) = \lambda b.$$
It is clear that $\|f\| = \|b\|$. % We will show that $B \in \mathcal{L}_{\gamma_{s, s_1,\dots, s_n}}^m(E_1,\dots, E_n; F)$.
Given $\left(x_j^{(r)} \right)_{j=1}^{\infty} \in \gamma_{s_r}(E_r), r=1,\dots, n$, it follows from linear stability of $\gamma_{s_r}$ that $\left(\varphi^{(r)}\left(x_j^{(r)}\right) \right)_{j=1}^{\infty} \in \gamma_{s_r}(\mathbb{K})$, $r=1,\dots, n$. Now, from  $\gamma_{s_1}(\mathbb{K})\cdots \gamma_{s_n}(\mathbb{K}) \overset{mult, 1}{\hookrightarrow} \gamma_s\left(\mathbb{K}; \mathbb{N}^n\right)$  it follows that
%Let's show that $f_i \in \mathcal{L}(\mathbb{K}; F)$ and $\|f_i\| = \|b_i\|$. The linearity is a direct consequence of the definition of $f_i$. The continuity follow of

%\begin{align*}
%\|f_i(\lambda)\| &= \|\lambda b_i\| \le \|b_i\| |\lambda|.
%\end{align*}
%Also we have that

%\begin{equation*}
%\|f_i\| = \sup_{|\lambda| \le 1}\|f(\lambda)\| \le \sup_{|\lambda|\le 1}\|b_1\||\lambda| = \|b_i\|,
%\end{equation*}
%Since
%\begin{equation*}
%\|f_i(1)\| = \|1 b_i\| = \|b_i\|.
%\end{equation*}
%Therefore $\|f_i\| = \|b_i\|$.
%As $\left(x_j^{(r)} \right)_{j=1}^{\infty} \in \gamma_{s_r}(E_r)$, $r=1,\dots, n$,

\begin{equation*}
\left(\varphi^{(1)}\left(x_{j_1}^{(1)}\right)\cdots \varphi^{(n)}\left(x_{j_n}^{(n)}\right) \right)_{j_1,\dots, j_n=1}^{\infty} \in \gamma_s\left(\mathbb{K}; \mathbb{N}^n\right).
\end{equation*}
Therefore, from linear stability of $\gamma_s\left(\cdot; \mathbb{N}^n \right)$ we have that
\begin{equation*}
\left(\varphi^{(1)}\left(x_{j_1}^{(1)}\right)\cdots \varphi^{(n)}\left(x_{j_n}^{(n)}\right)b \right)_{j_1,\dots, j_n=1}^{\infty} = \left(f\left(\varphi^{(1)}\left(x_{j_1}^{(1)}\right)\cdots \varphi^{(n)}\left(x_{j_n}^{(n)}\right) \right) \right)_{j_1,\dots, j_n=1}^{\infty} \in \gamma_s\left(F; \mathbb{N}^n\right),
\end{equation*}
%\begin{align*}
%\left\|\left(B_i\left(x_{j_1}^{(1)},\dots, x_{j_n}^{(n)} \right) \right)_{j_1,\dots, j_n=1}^{\infty} \right\|_{\gamma_s(F)} &= \left\|\left(\varphi_i^{(1)}(x_{j_1}^{(1)})\cdots \varphi_i^{(n)}(x_{j_n}^{(n)})b_i \right)_{j_1,\dots, j_n=1}^{\infty} \right\|_{\gamma_s(F)}\\
%&= \|b_i\| \left\|\left(\varphi_i^{(1)}(x_{j_1}^{(1)})\cdots \varphi_i^{(n)}(x_{j_n}^{(n)}) \right)_{j_1,\dots, j_n=1}^{\infty} \right\|_{\gamma_s(\mathbb{K})}\\
%&\le \|b_i\| \prod_{r=1}^n\left\|\left(\varphi_i^{(r)}\left(x_j^{(r)}\right) \right)_{j=1}^{\infty} \right\|_{\gamma_{s_r}(\mathbb{K})}\\
%&\le \|b_i\| \|\varphi_i^{(1)}\|\cdots \|\varphi_i^{(n)}\| \prod_{r=1}^n\left\|\left(x_j^{(r)} \right)_{j=1}^{\infty} \right\|_{\gamma_{s_r}(E_r)}.
%\end{align*}
what proves that $B \in \mathcal{L}_{\gamma_{s, s_1,\dots, s_n}}^m(E_1\dots, E_n; F)$. Since $\mathcal{L}_{\gamma_{s, s_1,\dots, s_n}}^m(E_1\dots, E_n; F)$
%the finite type operators are of the form
%
%
%\begin{equation*}
%T(x_1,\dots, x_n) = \sum_{i=1}^m\varphi_i^{(1)}(x_1)\cdots \varphi_i^{(n)}(x_n)b_i
%\end{equation*}
%and $\mathcal{L}_{\gamma_{s, s_1,\dots, s_n}}^m(E_1\dots, E_n; F)$
is a linear space, we conclude that it contains that the finite type operators.
\end{proof}

The next result proves the multi-ideal property.

\begin{proposition}\label{PPP3.5.}
Let $E_1,\dots, E_n, G_1,\dots, G_n, F$ and $H$ be Banach spaces, $t \in \mathcal{L}(F; H), T \in \mathcal{L}_{\gamma_{s, s_1,\dots, s_n}}^m(E_1,\dots, E_n; F)$ and $u_i \in \mathcal{L}(G_i; E_i), i=1,\dots, n$. Then

\begin{equation*}
t \circ T \circ (u_1,\dots, u_n) \in \mathcal{L}_{\gamma_{s, s_1,\dots, s_n}}^m(G_1,\dots, G_n; H)
\end{equation*}
and
\begin{equation*}
\| t \circ T \circ (u_1,\dots, u_n)\|_{\mathcal{L}_{\gamma_{s, s_1,\dots, s_n}}^m}\le \|t\|\| T\|_{\mathcal{L}_{\gamma_{s, s_1,\dots, s_n}}^m}\|u_1\|\cdots \|u_n\|.
\end{equation*}
\end{proposition}

\begin{proof}
%Let $t \in \mathcal{L}(F; H)$, $T \in \mathcal{L}_{\gamma_{s, s_1,\dots, s_n}}^m(E_1,\dots, E_n; F)$ and $u_i \in \mathcal{L}(G_i; E_i)$, $i=1,\dots, n$.
Let $\left(x_{j_i}^{(i)}\right)_{j_i=1}^{\infty} \in \gamma_{s_i}(G_i)$, $i=1,\dots, n$. Since each $\gamma_{s_i}$ is linearly stable, we have
 %and $T \in \mathcal{L}_{\gamma_{s, s_1,\dots, s_n}}^m(E_1,\dots, E_n; F)$, then for any $u_i \in \mathcal{L}(G_i; E_i)$, $i=1,\dots, n$
\begin{equation*}
\left(u_i\left(x_j^{(i)} \right) \right)_{j=1}^{\infty} \in \gamma_{s_i}(E_i)
\end{equation*}
for $i = 1, \ldots n$. As $\mathcal{L}_{\gamma_{s, s_1,\dots, s_n}}^m(E_1\dots, E_n; F)$, we get
\begin{equation*}
\left(T\circ (u_1,\dots, u_n)\left(x_{j_1}^{(1)},\dots, x_{j_n}^{(n)} \right) \right)_{j_1,\dots, j_n=1}^{\infty} \in \gamma_s\left(F; \mathbb{N}^n\right).
\end{equation*}
It follows from the linear stability of $\gamma_s\left(\cdot; \mathbb{N}^n \right)$ that %, for any $t \in \mathcal{L}(F; H)$
\begin{equation*}
\left(t \circ T \circ (u_1,\dots, u_n)\left(x_{j_1}^{(1)},\dots, x_{j_n}^{(n)} \right) \right)_{j_1,\dots, j_n=1}^{\infty}  \in \gamma_s\left(H; \mathbb{N}^n\right),
\end{equation*}
whenever $\left(x_{j_i}^{(i)}\right)_{j_i=1}^{\infty} \in \gamma_{s_i}(G_i)$, $i=1,\dots, n$. Therefore
\begin{equation*}
t \circ T \circ (u_1,\dots, u_n) \in \mathcal{L}_{\gamma_{s, s_1,\dots, s_n}}^m(G_1,\dots, G_n; H).
\end{equation*}
Thus, from linear stability of $\gamma_s\left(\cdot; \mathbb{N}^n \right)$ and $\gamma_{s_i}$, follows that
\begin{align*}
&\left\|\left( t \circ T(u_1,\dots, u_n)\left(x_{j_1}^{(1)},\dots, x_{j_n}^{(n)}\right) \right)_{j_1,\dots, j_n=1}^{\infty} \right\|_{\gamma_s\left(H; \mathbb{N}^n\right)}\\
&\le \|t\|\| T\|_{\mathcal{L}_{\gamma_{s, s_1,\dots, s_n}}^m}\|u_1\|\cdots \|u_n\|\prod_{i=1}^{n}\left\|\left(x_{j_i}^{(i)} \right)_{j_i=1}^{\infty} \right\|_{\gamma_{s_i}(E_i)},
\end{align*}
which completes the proof.
\end{proof}

Now we take the last step:%The next result will be very important to prove that $\left(\mathcal{L}_{\gamma_{s, s_1,\dots, s_n}}^m, \| \cdot\|_{\mathcal{L}_{\gamma_{s, s_1,\dots, s_n}}^m} \right)$ is a Banach Multi-ideal.

\begin{proposition}\label{PPP3.6.}
Consider the  map
 $$I_n : \mathbb{K}^n \rightarrow \mathbb{K}~,~I_n(\lambda_1,\dots, \lambda_n) = \lambda_1\cdots \lambda_n,$$ and suppose that $\gamma_{s_1}(\mathbb{K})\cdots \gamma_{s_n}(\mathbb{K}) \overset{mult, 1}{\hookrightarrow} \gamma_s\left(\mathbb{K}; \mathbb{N}^n\right)$. Then %$I_n \in \mathcal{L}_{\gamma_{s, s_1,\dots, s_n}}^m(\mathbb{K}^n; \mathbb{K})$ and
 $\| I_n\|_{\mathcal{L}_{\gamma_{s, s_1,\dots, s_n}}^m} = 1$.
\end{proposition}

\begin{proof}
  Given $\left(\lambda_{j_i}^{(i)}\right)_{j_i = 1}^{\infty} \in \gamma_{s_i}(\mathbb{K})$, $i=1,\dots, n$, it follows from $\gamma_{s_1}(\mathbb{K})\cdots \gamma_{s_n}(\mathbb{K}) \overset{mult, 1}{\hookrightarrow} \gamma_s\left(\mathbb{K}; \mathbb{N}^n\right)$ that $\left(\lambda_{j_1}^{(1)}\cdots \lambda_{j_n}^{(n)} \right)_{j_1,\dots, j_n=1}^{\infty} \in \gamma_s\left(\mathbb{K}; \mathbb{N}^n\right)$ and
\begin{align*}
\left\|\left(I_n\left(\lambda_{j_1}^{(1)} \cdots \lambda_{j_n}^{(n)}\right) \right)_{j_1 \cdots j_n=1}^{\infty} \right\|_{\gamma_s\left(\mathbb{K}; \mathbb{N}^n\right)} &= \left\| \left(\lambda_{j_1}^{(1)}\cdots \lambda_{j_n}^{(n)} \right)_{j_1,\dots, j_n=1}^{\infty} \right\|_{\gamma_s\left(\mathbb{K}; \mathbb{N}^n\right)} \le \prod_{i=1}^{n}\left\|\left(\lambda_{j}^{(i)} \right)_{j=1}^{\infty} \right\|_{\gamma_{s_i}(\mathbb{K})}.
\end{align*}
 So, $I_n \in \mathcal{L}_{\gamma_{s, s_1,\dots, s_n}}^m(\mathbb{K}^n; \mathbb{K})$ and $\| I_n\|_{\mathcal{L}_{\gamma_{s, s_1,\dots, s_n}}^m} \le 1$. The other inequality follows Proposition \ref{PPP3.3.}(b).
%Therefore $\| I_n\|_{\mathcal{L}_{\gamma_{s, s_1,\dots, s_n}}^m} = 1$.
%\begin{align*}
%1 &= \|I_n\|\\
%&= \sup_{\|\lambda_i\|\le 1}|I_n(\lambda_1,\dots, \lambda_n)|\\
%&\le \sup_{\left\|\left(\lambda_{j_i}^{(i)} \right)_{j_i=1}^{\infty} \right\|_{\gamma_{s_i}}}\left\|\left(I_n(\lambda_{j_1}^{(1)},\dots, \lambda_{j_n}^{(n)}) \right)_{j_1,\dots, j_n=1}^{\infty} \right\|_{\gamma_s(\mathbb{K})}\\
%&\le \| I_n\|_{\mathcal{L}_{\gamma_{s, s_1,\dots, s_n}}^m}\sup_{\left\|\left(\lambda_{j_i}^{(i)} \right)_{j_i=1}^{\infty} \right\|_{\gamma_{s_i}}}\prod_{i=1}^{n}\left\|\left(\lambda_{j_i}^{(i)} \right)_{j_i=1}^{\infty} \right\|_{\gamma_{s_i}(\mathbb{K})}\\
%&=  \| I_n\|_{\mathcal{L}_{\gamma_{s, s_1,\dots, s_n}}^m}.
%\end{align*}
%Hence the other inequality and
\end{proof}

Just combine Propositions \ref{PPP3.3.}, \ref{PPP3.4.}, \ref{PPP3.5.} and \ref{PPP3.6.} to obtain the following result:

\begin{theorem}\label{TTT3.7.}
Let  $\gamma_{s_1},\dots, \gamma_{s_n}$ be linearly stable, finitely determined sequence classes and $ \gamma_s\left(\cdot; \mathbb{N}^n \right)$ be a linearly stable, finitely determined $n$-sequence classes. Suppose that $\gamma_{s_1}\left(\mathbb{K} \right)\cdot \cdot \cdot \gamma_{s_n}\left(\mathbb{K} \right) \overset{mult, 1}{\hookrightarrow} \gamma_{s}\left(\mathbb{K}; \mathbb{N}^n \right)$. Then $\left(\mathcal{L}_{\gamma_{s, s_1,\dots, s_n}}^m, \| \cdot\|_{\mathcal{L}_{\gamma_{s, s_1,\dots, s_n}}^m} \right)$ is a Banach multi-ideal.
\end{theorem}
%The proof of the Theorem \ref{TTT3.7.} follows from .
It is worth mentioning that we used the condition $\gamma_{s_1}\left(\mathbb{K} \right)\cdot \cdot \cdot \gamma_{s_n}\left(\mathbb{K} \right) \overset{mult, 1}{\hookrightarrow} \gamma_{s}\left(\mathbb{K}; \mathbb{N}^n \right)$ twice, in the proofs of Propositions \ref{PPP3.5.} and \ref{PPP3.6.}.  %namely was essential for the proof of the Theorem \ref{TTT3.7.}, because it established a relation between the norm of the class $\gamma_{s_i}(\mathbb{K})$, $i=1,\dots, n$ and of the $n$-sequence class $\gamma_s\left(\mathbb{K}; \mathbb{N}^n\right)$, which was crucial to prove that the norm of identity is one. This requirement was also necessary to prove that finite type operators are contained in $\mathcal{L}_{\gamma_{s, s_1,\dots, s_n}}^m(E_1,\dots, E_n; F)$.
We also note that this condition %$\gamma_{s_1}\left(\mathbb{K} \right)\cdot \cdot \cdot \gamma_{s_n}\left(\mathbb{K} \right) \overset{mult, 1}{\hookrightarrow} \gamma_{s}\left(\mathbb{K}; \mathbb{N}^n \right)$
is not restrictive, since the main classes known in the literature enjoy this property.

In order to compare the class we are studying with the class of absolutely   $\gamma_{s, s_1,\dots, s_n}$-summing multilinear operators at the origin $a = (0,\dots, 0) \in E_1\times \cdots \times E_n$, which was introduced and denoted by $\prod_{\gamma_{s, s_1,\dots, s_n}}(E_1,\dots, E_n; F)$ in  \cite{S13}, we need the following definition.

%The propositions above we say that if $\gamma_{s_1}\left(\mathbb{K} \right)\cdot \cdot \cdot \gamma_{s_n}\left(\mathbb{K} \right) \overset{2}{\hookrightarrow} \gamma_{s}\left(\mathbb{K} \right)$, then  $\left(\mathcal{L}_{\gamma_{s, s_1,\dots, s_n}}^m, \| \cdot\|_{\mathcal{L}_{\gamma_{s, s_1,\dots, s_n}}^m} \right)$ is a Banach Ideal.

\begin{definition}\label{DDD.2.10.} For every $n \in \mathbb{N}$, let an $n$-sequence class $\gamma_s\left(\cdot; \mathbb{N}^n \right)$ be given. %an $n$-sequence class, for every $n \in \mathbb{N}$ and $E$ be Banach space.
We say that the $n$-sequence class $\gamma_s\left(\cdot; \mathbb{N}^n \right)$ is sequentially compatible with $\gamma_s(\cdot) := \gamma_s\left(\cdot; \mathbb{N}^1 \right)$ if, for every $(x_{j_1,\dots, j_n})_{j_1,\dots, j_n=1}^{\infty} \in \gamma_s\left(E; \mathbb{N}^n \right)$, we have $(x_j)_{j=1}^{\infty} := (x_{j,\dots, j})_{j=1}^{\infty} \in \gamma_s(E)$ and
\begin{equation*}
\left\|(x_{j})_{j=1}^{\infty} \right\|_{\gamma_s(E)} \le \left\|(x_{j_1,\dots, j_n})_{j_1,\dots, j_n = 1}^{\infty} \right\|_{\gamma_s\left(E; \mathbb{N}^n \right)}.
\end{equation*}
\end{definition}

\begin{example}\label{EEEE.2.11.}
All classes mentioned in Example \ref{ExSc} satisfy Definition \ref{DDD.2.10.}.
\end{example}

%The next proposition relates the concept of multiple $\gamma_{s, s_1,\dots, s_n}$-summing operators with the absolutely   $\gamma_{s, s_1,\dots, s_n}$-summing multilinear operator in $a = (0,\dots, 0) \in E_1\times \cdots \times E_n$, when consider that the sequence class $\gamma_s\left(\cdot, \mathbb{N}^n \right)$ have the property of be sequentially compatible with $\gamma_s(\cdot)$, described in definition \ref{DDD.2.10.}. The absolutely   $\gamma_{s, s_1,\dots, s_n}$-summing multilinear operator was introduced in literature in \cite{S13}. In this work, the absolutely   $\gamma_{s, s_1,\dots, s_n}$-summing multilinear operator in $a = (0,\dots, 0) \in E_1\times \cdots \times E_n$ was denoted by $\prod_{\gamma_{s, s_1,\dots, s_n}}(E_1,\dots, E_n; F)$.
\begin{proposition}
	Let $E_1,\dots, E_n, F$ be Banach spaces and suppose that $\gamma_s\left(\cdot; \mathbb{N}^n \right)$ sequentially compatible with $\gamma_s(\cdot)$. Then $\mathcal{L}_{\gamma_{s, s_1,\dots, s_n}}^m(E_1,\dots, E_n; F) \subset \prod_{\gamma_{s, s_1,\dots, s_n}}(E_1,\dots, E_n; F)$ and
		\begin{equation*}
	\pi_{\gamma_{s, s_1,\dots, s_n}}(T) \le \|T\|_{\mathcal{L}_{\gamma_{s, s_1,\dots, s_n}}^m}
	\end{equation*}
	for any $T \in \mathcal{L}_{\gamma_{s, s_1,\dots, s_n}}^m(E_1,\dots, E_n; F)$.
\end{proposition}

\begin{proof}
	Given $T \in \mathcal{L}_{\gamma_{s, s_1,\dots, s_n}}^m(E_1,\dots, E_n; F)$ and $\left(x_j^{(i)} \right)_{j=1}^{\infty} \in \gamma_{s_i}(E_i), i=1,\dots, n$, it follows directly from Definition \ref{DDD.2.10.}  that $\left(T\left(x_{j}^{(1)},\dots, x_j^{(n)} \right) \right)_{j=1}^{\infty} \in \gamma_s(F) $, which gives $T \in \prod_{\gamma_{s, s_1,\dots, s_n}}(E_1,\dots, E_n; F)$, and
	\begin{equation*}
	\left\|\left(T\left(x_j^{(1)},\dots, x_j^{(n)} \right) \right)_{j=1}^{\infty} \right\|_{\gamma_s\left(F; \mathbb{N}^n\right)} \le \|T\|_{\mathcal{L}_{\gamma_{s, s_1,\dots, s_n}}^m} \prod_{i=1}^n\left\|\left(x_j^{(i)} \right)_{j=1}^{\infty} \right\|_{\gamma_{s_i}(E_i)},
	\end{equation*}
%	\begin{align*}
%	\left\|\left(T\left(x_j^{(1)},\dots, x_j^{(n)} \right) \right)_{j=1}^{\infty} \right\|_{\gamma_s\left(F; \mathbb{N}^n\right)} &\le \left\|\left(T\left(x_{j_1}^{(1)},\dots, x_{j_n}^{(n)} \right) \right)_{j_1,\dots, j_n=1}^{\infty} \right\|_{\gamma_s\left(F; \mathbb{N}^n\right)} \le \|T\|_{\mathcal{L}_{\gamma_{s, s_1,\dots, s_n}}^m} \prod_{i=1}^n\left\|\left(x_j^{(i)} \right)_{j=1}^{\infty} \right\|_{\gamma_{s_i}(E_i)}.
%	\end{align*}
 which proves that $\pi_{\gamma_{s, s_1,\dots, s_n}}(T) \le \|T\|_{\mathcal{L}_{\gamma_{s, s_1,\dots, s_n}}^m}$.
\end{proof}

We finish this section giving another consequence of definition above.

\begin{proposition}
Let  $\gamma_{s_1},\dots, \gamma_{s_n}$ be finitely determined sequence classes, $\gamma_s\left(\cdot; \mathbb{N}^n \right)$ be a finitely determined $n$-sequence class and suppose that $\gamma_s\left(\cdot; \mathbb{N}^n \right)$ sequentially compatible with $\gamma_s(\cdot)$. \,If $\gamma_{s_1}(\mathbb{K})\cdots \gamma_{s_n}(\mathbb{K}) \overset{mult, 1}{\hookrightarrow} \gamma_s\left(\mathbb{K}; \mathbb{N}^n\right)$, then $\gamma_{s_1}(\mathbb{K})\cdots \gamma_{s_n}(\mathbb{K}) \overset{1}{\hookrightarrow} \gamma_s(\mathbb{K})$.

%	Given finitely determined sequence classes $\gamma_{s_1},\dots, \gamma_{s_n}$ and finitely determined $n$-sequence class $\gamma_s\left(\cdot; \mathbb{N}^n \right)$. If $\gamma_{s_1}(\mathbb{K})\cdots \gamma_{s_n}(\mathbb{K}) \overset{mult, 1}{\hookrightarrow} \gamma_s\left(\mathbb{K}; \mathbb{N}^n\right)$, then, $\gamma_{s_1}(\mathbb{K})\cdots \gamma_{s_n}(\mathbb{K}) \overset{1}{\hookrightarrow} \gamma_s(\mathbb{K})$.
\end{proposition}
\begin{proof}
	Given $\left(\lambda_j^{(i)} \right)_{j=1}^{\infty} \in \gamma_{s_i}(\mathbb{K})$, $i=1,\dots, n$, by Definition \ref{DDD.2.10.} we have $\left(\lambda_j^{(1)}\cdots \lambda_j^{(n)} \right)_{j=1}^{\infty} \in \gamma_s(\mathbb{K})$ and
	
	%%As $\gamma_{s_1}(\mathbb{K})\cdots \gamma_{s_n}(\mathbb{K}) \overset{mult, 1}{\hookrightarrow} \gamma_s(\mathbb{K})$, then $\left(\lambda_{j_1}^{(1)}\cdots \lambda_{j_n}^{(n)} \right)_{j_1,\dots, j_n=1}^{\infty} \in \gamma_s(\mathbb{K})$ and
	%%\begin{equation*}
	%%\left\|\left(\lambda_{j_1}^{(1)}\cdots \lambda_{j_n}^{(n)} \right)_{j_1,\dots, j_n=1}^{\infty} \right\|_{\gamma_s(\mathbb{K})} \le \prod_{i=1}^{\infty}\left\|\left(\lambda_j^{(i)} \right)_{j=1}^{\infty} \right\|_{\gamma_{s_i}(\mathbb{K})}.
	%%\end{equation*}
	%%Thus, how $\gamma_{s_1},\dots, \gamma_{s_n}, \gamma_s$ are finitely determined
	%\begin{equation*}
	%\sup_m\left\|\left(\lambda_j^{(1)}\cdots \lambda_j^{(n)} \right)_{j=1}^m \right\|_{\gamma_s(\mathbb{K})} \le \sup_m\left\|\left(\lambda_{j_1}^{(1)}\cdots \lambda_{j_n}^{(n)} \right)_{j_1,\dots, j_n=1}^m \right\|_{\gamma_s(\mathbb{K})}.
	%\end{equation*}
	%Since $\gamma_{s_1}(\mathbb{K})\cdots \gamma_{s_n}(\mathbb{K}) \overset{mult, 1}{\hookrightarrow} \gamma_s(\mathbb{K})$ it follows that
	%\begin{equation*}
	%\sup_m\left\|\left(\lambda_j^{(1)}\cdots \lambda_j^{(n)} \right)_{j=1}^m \right\|_{\gamma_s(\mathbb{K})} < \infty
	%\end{equation*}
	%and, how $\gamma_s$ is finitely determined, we obtain that $\left(\lambda_j^{(1)}\cdots \lambda_j^{(n)} \right)_{j=1}^{\infty} \in \gamma_s(\mathbb{K})$. Futhermore, we have
	\begin{equation*}
	\left\|\left(\lambda_{j}^{(1)}\cdots \lambda_{j}^{(n)} \right)_{j=1}^{\infty} \right\|_{\gamma_s(\mathbb{K})} \le \left\|\left(\lambda_{j_1}^{(1)}\cdots \lambda_{j_n}^{(n)} \right)_{j_1,\dots, j_n=1}^{\infty} \right\|_{\gamma_s\left(\mathbb{K}; \mathbb{N}^n\right)} \le \prod_{i=1}^{\infty}\left\|\left(\lambda_j^{(i)} \right)_{j=1}^{\infty} \right\|_{\gamma_{s_i}(\mathbb{K})}.
	\end{equation*}
		%Therefore, $\gamma_{s_1}(\mathbb{K})\cdots \gamma_{s_n}(\mathbb{K}) \overset{1}{\hookrightarrow} \gamma_s(\mathbb{K})$.
\end{proof}

\section{Coherence and compatibility}\label{Sec4}

The concept of coherence and compatibility was introduced in the literature initially by D. Pellegrino and G. Botelho in \cite{30, 31} (with a different terminology) and also was studied by D. Carando, V. Dimant and S. Muro in \cite{27, 29, 28}. The terms "coherence and compatibility" were coined in \cite{27}. In \cite{13} it was introduced a new approach to ''coherence and compatibility'', which considers the sequence formed by the pairs of ideals of multilinear applications and homogeneous polynomials. For this reason, some information about homogeneous polynomials is needed. %it is necessary to define the class of the homogeneous polynomials.% associated with the multilinear applications.

The class of all continuous homogeneous polynomials between Banach spaces is denoted by $\cal P$. Given an $n$-homogeneous polynomial $P \colon E \longrightarrow F$, by $\check P$ we denote the unique symmetric continuous $n$-linear operator associated to $P$. For any unexplained notation about polynomials we refer to \cite{B05, 13}. The next result is folklore.

\begin{proposition}\label{ppp}%{\cite[Proposition 2.1.13]{35}}
Let $\left(\mathcal{M}, \|\cdot\|_{\mathcal{M}}\right)$ be a Banach ideal of multilinear operators. Then, the class
\begin{equation*}
\mathcal{P}_{\mathcal{M}} := \left\{P \in \mathcal{P}; \check{P} \in \mathcal{M} \right\}~,~\|P\|_{\mathcal{P}_{\mathcal{M}}} := \|\check{P}\|_{\mathcal{M}},
\end{equation*}
is a Banach ideal of homogeneous polynomials.
\end{proposition}

The Banach ideal $\left(\mathcal{P}_{\mathcal{M}}, \|\cdot\|_{\mathcal{P}_{\mathcal{M}}}\right)$ of homogeneous polynomials is called the Banach ideal of polynomials generated by $\left(\mathcal{M}, \|\cdot\|_{\mathcal{M}}\right)$. This class has been studied extensively in several works, of which we highlight \cite{B05, 16, 17}.

In this section, the class of the multiple $\gamma_{s, s_1, \ldots, s_1}$-summing $n$-linear operators shall be denoted by $\mathcal{L}_{\gamma_{s, s_1}}^{m, n}$. %The reason for this is to evidence the linearity of the components of the ideal.

%In the next definition, we will presented the concept of multiple $\gamma_{s, s_1}$-summing homogeneous polynomials.

%The concept of Coherence and Compatibility was introduced in the literature for Pellegrino and Ribeiro in \cite{13}. This concept considers the sequence formed by the pairs of ideals of multilinear applications and homogeneous polynomials. For this reason, it is necessary to define the class of the homogeneous polynomials associated with the multilinear applications.

%Let's define the multiple  $\gamma_{s, s_1}$-summing homogeneos polynomious in the next definition.

\begin{definition}\label{DDD4.1.}
Given $E$ and $F$ Banach spaces, the class of the multiple $\gamma_{s, s_1}$-summing $n$-homogeneos polynomials is defined by
\begin{equation*}
\mathcal{P}_{\gamma_{s, s_i}}^{m, n} :=\mathcal{P}_{\mathcal{L}_{\gamma_{s, s_i}}^{m, n}} := \left\{P \in \mathcal{P}; \check{P} \in \mathcal{L}_{\gamma_{s, s_i}}^{m, n} \right\}.
\end{equation*}
%where $\mathcal{L}_{\gamma_{s, s_i}}^{m, n}:=\mathcal{L}_{\gamma_{s,s_i,\dots,s_i}}^{m, n}$.
\end{definition}

%In the papers \cite{15, 16, 17} we find the following result.

%The following result is folcloric.

Since $\left(\mathcal{L}_{\gamma_{s, s_1}}^{m, n} ,\left\|\cdot \right\|_{\mathcal{L}_{\gamma_{s, s_1}}^{m,n}}\right)$ is Banach ideal of multilinear operators, from Proposition \ref{ppp} we have the following result.
%An analogous result for the multiple ideal is valid. The proof is a simple adaptation of the known case and will be omitted.

\begin{corollary}
Let $E$ and $F$ be Banach spaces. Then $\left(\mathcal{P}_{\gamma_{s, s_1}}^{m, n}, \left\|P \right\|_{\mathcal{P}_{\gamma_{s, s_1}}^{m, n}} \right)$ is a Banach ideal of $n$-homogeneous polynomials endowed with the norm

\begin{equation*}
\left\|P \right\|_{\mathcal{P}_{\gamma_{s, s_1}}^{m, n}} := \left\|\check{P} \right\|_{\mathcal{L}_{\gamma_{s, s_1}}^{m,n}}.
\end{equation*}

\end{corollary}

The definition below shall prove to be the correct condition the sequence classes should satisfy for coherence to hold.%, it is necessary require that the sequence classes satisfy the following conditions.

\begin{definition} A sequence $\left(\gamma_n\left(\cdot; \mathbb{N}^{n} \right)\right)_{n=1}^{M}$, where $M \in \mathbb{N}\cup \{\infty \}$ and each $\gamma_n\left(\cdot; \mathbb{N}^{n}\right)$ is an $n$-sequence class, is said to be:\\
(a) {\bf multiple regular} with the sequence class $ \gamma_{s_i} $ if the following condition holds: if $\left(\lambda_{j} \right)_{j=1}^{\infty} \in \gamma_{s_i}(\mathbb{K})$, $i=1,\dots, n$,  and $\left(a_{j_1,\dots, j_{i-1}, j_{i+1},\dots, j_n}\right)_{j_1,\dots, j_{i-1}, j_{i+1},\dots, j_n=1}^{\infty}  \in \gamma_{n-1}\left(F; \mathbb{N}^{n-1}\right)$, regardless of the Banach space $F$, then $\left(\lambda_{j_{i}}a_{j_1,\dots, j_{i-1}, j_{i+1},\dots, j_n}\right)_{j_1,\dots, j_{n}=1}^{\infty} \in \gamma_n\left(F; \mathbb{N}^{n}\right)$ and
\begin{align*}
&\left\|\left(\lambda_{j_{i}}a_{j_1,\dots, j_{i-1}, j_{i+1},\dots, j_n}\right)_{j_1,\dots, j_{n}=1}^{\infty} \right\|_{\gamma_{n}\left(F; \mathbb{N}^{n}\right)} \\
&\le \left\|\left(\lambda_{j} \right)_{j=1}^{\infty} \right\|_{\gamma_{s_i}(\mathbb{K})} \left\| \left(a_{j_1,\dots, j_{i-1}, j_{i+1},\dots, j_{n}}\right)_{j_1,\dots, j_{i-1}, j_{i+1},\dots, j_{n}=1}^{\infty} \right\|_{\gamma_{n-1}\left(F; \mathbb{N}^{n-1}\right)}.
\end{align*}
(b) %Let $\gamma_n\left(\cdot; \mathbb{N}^n \right)$ be a $n$-sequence class and $E$ be a Banach space. We say that the sequence of $n$-sequence classes $\left(\gamma_n\left(\cdot; \mathbb{N}^n \right) \right)_{n=1}^{M}$, $M \in \mathbb{N}\cup \{\infty \}$ is
{\bf down regular} if, for any Banach space $E$ and every $(x_{j_1,\dots, j_n})_{j_1,\dots, j_n=1}^{\infty} \in \gamma_n\left(E; \mathbb{N}^n\right)$ with $n \ge 2$, for any fixed $j_i$,  $i=1,\dots, n$, it holds that $(x_{j_1,\dots, j_n})_{j_1,\dots, j_{i-1}, j_{i+1},\dots, j_n = 1}^{\infty} \in \gamma_{n-1}\left(E; \mathbb{N}^{n-1}\right)$ and
\begin{align*}
&\left\|(x_{j_1,\dots, j_n})_{j_1,\dots, j_{i-1}, j_{i+1},\dots, j_n = 1}^{\infty} \right\|_{\gamma_{n-1}\left(E; \mathbb{N}^{n-1}\right)} \le \left\|(x_{j_1,\dots, j_n})_{j_1,\dots, j_n=1}^{\infty} \right\|_{\gamma_n\left(E; \mathbb{N}^n\right)}.
\end{align*}
\end{definition}

\begin{example}\label{Exemple4.6}
	All classes presented in Example \ref{ExSc} are multiple regular and down regular.
\end{example}

Besides of guaranteeing coherence, as we shall prove soon, the definitions above avoid artificial sequences of $n$-sequences, as the following example illustrates.

\begin{example}
	Let $\gamma_n\left(\cdot; \mathbb{N}^n \right)$ be defined by: $\gamma_n\left(\cdot; \mathbb{N}^n \right):=\ell_p^w\left(\cdot; \mathbb{N}^n\right)$ if $n$ is even and $\gamma_n\left(\cdot; \mathbb{N}^n \right):=\ell_p\left(\cdot; \mathbb{N}^n\right)$ if $n$ is odd. It is plain that the sequence $\left(\gamma_n\left(\cdot; \mathbb{N}^n \right) \right)_{n=1}^{M}$ is neither multiple regular nor down regular.
\end{example}

A classic result that will be important to  next Theorem is as following.

\begin{lemma}\label{LLL4.4.}
Let $P \in \mathcal{P}(^nE; F)$ and $a \in E$. Then $(P_a)^{\vee} = \check{P}_a$.
\end{lemma}

The main result of this sections reads as follows.

\begin{theorem}\label{ttt}
Let $\gamma_{s_i}$ be a finitely determined and linearly stable sequence class and, for every $n \in \mathbb{N}$, let $\gamma_{s}\left(\cdot, \mathbb{N}^n \right)$ be a finitely determined and linearly stable $n$-sequence class. Suppose that the sequence of $n$-sequence classes $\left(\gamma_s\left(\cdot; \mathbb{N}^{n} \right)\right)_{n=1}^{\infty}$ is multiple regular with $ \gamma_{s_i} $ and down regular. Then the sequence of pairs $$\left(\left( \mathcal{L}_{\gamma_{s, s_1,\dots,s_n}}^{m, n}, \| \cdot \|_{\mathcal{L}_{\gamma_{s, s_1,\dots,s_n}}^{m, n}} \right); \left(\mathcal{P}_{\gamma_{s, s_i}}^{m,n}, \| \cdot\|_{\mathcal{P}_{\gamma_{s, s_i}}^{m,n}} \right) \right)_{n=1}^{\infty}$$ is coherent and compatible with $\mathcal{L}_{\gamma_{s, s_i}}$ in the sense of \cite{13}.
\end{theorem}

\begin{proof}

To prove (CH1), we will do only the case $i=1$. The general case is analogous. Let $T \in \mathcal{L}_{\gamma_{s, s_1,\dots,s_{n+1}}}^{m, n+1}(E_1,\dots, E_{n+1}; F)$ and $a_1 \in E_{1}$. 
%Faremos a prova para $i=1$. O caso geral se faz de forma an\'aloga. Sejam $T \in \mathcal{L}_{\gamma_{s, s_1}}^{n+1}(E_1,\dots, E_{n+1}; F)$ e $a_1 \in E_{1}$. Assim

Consider the sequence $\left(x_{j}^{(1)} \right)_{j=1}^{\infty}$, such that, $x_1^{(1)}=a_1$ and $x_{j}^{(1)}=0$ for $j \neq 1$. It is easy to see that, $\left(x_{j}^{(1)} \right)_{j=1}^{\infty} \in \gamma_{s_1}(E_1)$. Let $\left(x_{j}^{(i)} \right)_{j=1}^{\infty} \in \gamma_{s_i}(E_i)$, $i=2,\dots, n+1$. Thus $$\left(T\left(x_{j_1}^{(1)},\dots, x_{j_{n+1}}^{(n+1)} \right) \right)_{j_1,\dots, j_{n+1}=1}^{\infty} \in \gamma_s\left(F; \mathbb{N}^{n+1}\right).$$ Assuming $j_1=1$, since $\left(\gamma_n\left(\cdot; \mathbb{N}^{n} \right)\right)_{n=1}^{\infty}$ is down regular and the sequence class $\gamma_{s_1}$ is linearly stable, we have
\begin{equation*}
\left(T_{a_1}\left(x_{j_2}^{(2)},\dots, x_{j_{n+1}}^{(n+1)} \right) \right)_{j_2,\dots, j_{n+1}=1}^{\infty} = \left(T\left(a_1, x_{j_2}^{(2)},\dots, x_{j_{n+1}}^{(n+1)} \right) \right)_{j_2,\dots, j_{n+1}=1}^{\infty} \in \gamma_s\left(F; \mathbb{N}^{n}\right)
\end{equation*}
and
\begin{align*}
\left\|\left(T_{a_1}\left(x_{j_2}^{(2)},\dots, x_{j_{n+1}}^{(n+1)} \right) \right)_{j_2,\dots, j_{n+1}=1}^{\infty} \right\|_{\gamma_s\left(F; \mathbb{N}^{n}\right)} %&=\left\|\left(T\left(a_1, x_{j_2}^{(2)},\dots, x_{j_{n+1}}^{(n+1)} \right) \right)_{j_2,\dots, j_{n+1}=1}^{\infty} \right\|_{\gamma_s\left(F; \mathbb{N}^{n}\right)}\\
&\le \left\| \left(T\left(x_{j_1}^{(1)},\dots, x_{j_{n+1}}^{(n+1)} \right) \right)_{j_1,\dots, j_{n+1}=1}^{\infty}\right\|_{\gamma_s\left(F; \mathbb{N}^{n+1}\right)}\\
%&\le \|T\|_{\mathcal{L}_{\gamma_{s, s_1,\dots,s_{n+1}}}^{m, n+1}}\prod_{i=1}^{n+1}\left\|\left(x_{j}^{(i)} \right)_{j=1}^{\infty} \right\|_{\gamma_{s_1}(E_i)}\\
&\le \|T\|_{\mathcal{L}_{\gamma_{s, s_1,\dots,s_{n+1}}}^{m, n+1}}\|a_1\| \prod_{i=2}^{n+1}\left\|\left(x_{j}^{(i)} \right)_{j=1}^{\infty} \right\|_{\gamma_{s_1}(E_i)}.
\end{align*}

%\begin{align*}
%\left\|\left(T_{a_1}\left(x_{j_2}^{(2)},\dots, x_{j_{n+1}}^{(n+1)}\right) \right)_{j_2,\dots, j_n=1}^{\infty}  \right\|_{\gamma_s(F)} &= \left\|\left(T\left(a_1, x_{j_2}^{(2)},\dots, x_{j_{n+1}}^{(n+1)}\right) \right)_{j_2,\dots, j_n=1}^{\infty}  \right\|_{\gamma_s(F)}\\
%&\overset{\text{Lemma \ref{LLL1.5.}}}{=} \left\|\left(T\left(x_{j_1}^{(1)}, x_{j_2}^{(2)},\dots, x_{j_{n+1}}^{(n+1)}\right) \right)_{j_1,\dots, j_n=1}^{\infty}  \right\|_{\gamma_s(F)}\\
%&\le \|T\|_{\mathcal{L}_{\gamma_{s, s_1}}^{m, n+1}}\|a_1\|\prod_{i=2}^n\left\|\left(x_{j_i}^{(i)} \right)_{j_i=1}^{\infty} \right\|_{\gamma_{s_1}(E_i)}.
%\end{align*}
Therefore, $T_{a_1} \in \mathcal{L}_{\gamma_{s, s_2,\dots,s_{n+1}}}^{m, n}(E_2,\dots, E_{n+1}; F)$ and $\|T_{a_1}\|_{\mathcal{L}_{\gamma_{s, s_2,\dots,s_{n+1}}}^{m, n}} \le \|T\|_{\mathcal{L}_{\gamma_{s, s_1,\dots,s_{n+1}}}^{m, n+1}}\|a_1\|.$

Now we will check (CH3). Let $T \in \mathcal{L}_{\gamma_{s, s_1,\dots,s_{n}}}^{m, n}(E_1,\dots, E_{n}; F)$ and $\gamma \in E_{n+1}'$. Let $\left(x_{j}^{(i)} \right)_{j=1}^{\infty} \in \gamma_{s_i}(E_i)$, $i=1,\dots, n+1$. Since $\left(\gamma_n\left(\cdot; \mathbb{N}^{n} \right)\right)_{n=1}^{\infty}$ is multiple regular with $\gamma_{s_i}(\cdot)$ and the sequence classes are linearly stable, then $\left(\gamma T\left(x_{j_1}^{(1)},\dots, x_{j_{n+1}}^{(n+1)}\right) \right)_{j_1,\dots, j_{n+1}=1}^{\infty} = \left(\gamma\left(x_{j_{n+1}}^{(n+1)}\right) T\left(x_{j_1}^{(1)},\dots, x_{j_{n}}^{(n)}\right) \right)_{j_1,\dots, j_{n+1}=1}^{\infty} \in \gamma_s\left(F; \mathbb{N}^{n+1}\right)$ and
\begin{align*}
&\left\|\left(\gamma T\left(x_{j_1}^{(1)},\dots, x_{j_{n+1}}^{(n+1)}\right) \right)_{j_1,\dots, j_{n+1}=1}^{\infty}  \right\|_{\gamma_s\left(F; \mathbb{N}^{n+1}\right)}\\
%&= \left\|\left(\gamma\left(x_{j_{n+1}}^{(n+1)}\right) T\left(x_{j_1}^{(1)},\dots, x_{j_{n}}^{(n)}\right) \right)_{j_1,\dots, j_{n+1}=1}^{\infty}  \right\|_{\gamma_s\left(F; \mathbb{N}^{n+1}\right)}\\
&\le \left\|\left(\gamma\left(x_{j}^{(n+1)}\right)  \right)_{j=1}^{\infty} \right\|_{\gamma_{s_{n+1}}(\mathbb{K})}\left\| \left(T\left(x_{j_1}^{(1)},\dots, x_{j_{n}}^{(n)} \right) \right)_{j_1,\dots, j_n=1}^{\infty} \right\|_{\gamma_s\left(F; \mathbb{N}^n\right)}\\
%&\le\|\gamma \|\left\|\left(x_{j}^{(n+1)} \right)_{j=1}^{\infty} \right\|_{\gamma_{s_{n+1}(E_{n+1})}} \left\|\left(T\left(x_{j_1}^{(1)},\dots, x_{j_{n}}^{(n)} \right) \right)_{j_1,\dots, j_n=1}^{\infty} \right\|_{\gamma_s\left(F; \mathbb{N}^n\right)}\\
&\le \|\gamma \| \|T\|_{\mathcal{L}_{\gamma_{s, s_1,\dots,s_{n}}}^{m, n}} \prod_{i=1}^{n+1}\left\|\left(x_{j}^{(i)} \right)_{j=1}^{\infty} \right\|_{\gamma_{s_i}(E_i)}.
\end{align*}
Therefore, $\gamma T \in \mathcal{L}_{\gamma_{s, s_1,\dots,s_{n+1}}}^{m, n+1}(E_1,\dots, E_{n+1}; F)$ and $\|\gamma T\|_{\mathcal{L}_{\gamma_{s, s_1,\dots,s_{n+1}}}^{m, n+1}} \le \|\gamma \| \|T\|_{\mathcal{L}_{\gamma_{s, s_1,\dots,s_{n}}}^{m, n}}.$

Now we will prove (CH2). Let $P \in \mathcal{P}_{\gamma_{s, s_i}}^{m, n+1}(^{n+1}E; F)$ and $a \in E$. To see that $P_{a} \in \mathcal{P}_{\gamma_{s, s_i}}^{m, n}(^nE; F)$ is enough to show that $(P_a)^{\vee}\in \mathcal{L}_{\gamma_{s, s_i}}^{m, n}(E^n; F)$. Since $P \in \mathcal{P}_{\gamma_{s, s_i}}^{n+1}(^{n+1}E; F)$, then $$\check{P} \in \mathcal{L}_{\gamma_{s, s_i}}^{m, n+1}(E^{n+1}; F),$$ thus, by (CH1), $$\check{P}_a \in \mathcal{L}_{\gamma_{s, s_i}}^{m, n}(E^{n}; F).$$ By the Lemma \ref{LLL4.4.}, we have that $$(P_a)^{\vee} = \check{P}_a \in \mathcal{L}_{\gamma_{s, s_i}}^{m, n}(E^{n}; F).$$ Like this, 
\begin{equation*}
\|P_{a}\|_{\mathcal{P}_{\gamma_{s, s_i}}^{m, n}} = \|(P_{a})^{\vee}\|_{\mathcal{L}_{\gamma_{s, s_i}}^{m, n}} = \|\check{P}_a\|_{\mathcal{L}_{\gamma_{s, s_i}}^{m, n}} \le \|\check{P}\|_{\mathcal{L}_{\gamma_{s, s_i}}^{m, n+1}}\|a\|.
\end{equation*}

Now we will prove (CH4). Let $P \in \mathcal{P}_{\gamma_{s, s_i}}^{m, n}(^nE; F)$ and $\varphi \in E'$. As done in (CH2), to see that $\varphi P \in \mathcal{P}_{\gamma_{s, s_i}}^{m, n+1}(^{n+1}E; F)$ is enough to show that $(\varphi P)^{\vee} \in \mathcal{L}_{\gamma_{s, s_i}}^{m, n+1}(E^{n+1}; F)$. Note that

\begin{equation*}
(\varphi P)^{\vee}\left(x_{j_1}^{(1)},\dots, x_{j_{n+1}}^{(n+1)} \right) = \displaystyle\frac{\varphi\left(x_{j_1}^{(1)}\right)\check{P}\left(x_{j_2}^{(2)} , \dots, x_{j_{n+1}}^{(n+1)}\right) + \cdots + \varphi\left(x_{j_{n+1}}^{(n+1)}\right)\check{P}\left(x_{j_1}^{(1)} , \dots, x_{j_n}^{(n)}\right)}{(n+1)!}.
\end{equation*}
Since $P \in \mathcal{P}_{\gamma_{s, s_i}}^{m, n}(^nE; F)$, then $\check{P} \in \mathcal{L}_{\gamma_{s, s_i}}^{m, n}(E^n; F)$. Thus, for any $\left(x_j^{(k)} \right)_{j=1}^{\infty} \in \gamma_{s_i}(E)$, $k=1,\dots, n+1$

\begin{equation*}
\left(\check{P}\left(x_{j_2}^{(2)} , \dots, x_{j_{n+1}}^{(n+1)}\right)\right)_{j_2,\dots,j_{n+1}=1}^{\infty},\dots,\left(\check{P}\left(x_{j_1}^{(1)},\dots, x_{j_n}^{(n)} \right) \right)_{j_1,\dots, j_n=1}^{\infty} \in \gamma_s\left(F; \mathbb{N}^n\right).
\end{equation*}
So, $\left(\gamma_n\left(\cdot; \mathbb{N}^{n} \right)\right)_{n=1}^{\infty}$ is multiple regular with $\gamma_{s_i}(\cdot)$, then for any $\left(x_j^{(k)} \right)_{j=1}^{\infty} \in \gamma_{s_i}(E)$, $k=1,\dots, n+1$

\begin{equation*}
\left(\varphi\left(x_{j_k}^{(k)} \right) \check{P}\left(x_{j_1}^{(1)},\dots, x_{j_{k-1}}^{(k-1)}, x_{j_{k+1}}^{(k+1)},\dots, x_{j_{n+1}}^{(n+1)} \right) \right)_{j_1,\dots, j_{n+1}=1}^{\infty} \in \gamma_s\left(F; \mathbb{N}^{n+1}\right).
\end{equation*}
Therefore
\begin{align*}
&\left((\varphi P)^{\vee}\left(x_{j_1}^{(1)},\dots, x_{j_{n+1}}^{(n+1)} \right)\right)_{j_1,\dots, j_{n+1}=1}^{\infty}\\ 
&= \left(\displaystyle\frac{\varphi\left(x_{j_1}^{(1)}\right)\check{P}\left(x_{j_2}^{(2)} , \dots, x_{j_{n+1}}^{(n+1)}\right) + \cdots + \varphi\left(x_{j_{n+1}}^{(n+1)}\right)\check{P}\left(x_{j_1}^{(1)} , \dots, x_{j_n}^{(n)}\right)}{n+1}  \right)_{j_1,\dots, j_{n+1}=1}^{\infty} \in \gamma_s\left(F; \mathbb{N}^{n+1}\right).
\end{align*}
Like this, $(\varphi P)^{\vee} \in \mathcal{L}_{\gamma_{s, s_i}}^{m, n+1}(E^{n+1}; F)$. Note also that, by (CH3)
\begin{equation*}
\|\varphi P\|_{\mathcal{P}_{\gamma_{s, s_i}}^{m, n+1}} = \|\left(\varphi P \right)^{\vee}\|_{\mathcal{L}_{\gamma_{s, s_i}}^{m, n+1}} \le \|\varphi \check{P}\|_{\mathcal{L}_{\gamma_{s, s_i}}^{m, n+1}} \le \|\varphi\| \|\check{P}\|_{\mathcal{L}_{\gamma_{s, s_i}}^{m, n+1}} = \|\varphi\| \|P\|_{\mathcal{P}_{\gamma_{s, s_i}}^{m, n+1}}.
\end{equation*}

The condition (CH5) follows from Definition \ref{DDD4.1.}.

\end{proof}

For the definition of global holomorphy type, see, e.g., \cite{30}.

\begin{corollary}\label{T5.3.} Under the assumptions of Theorem \ref{ttt},
$ \left(\mathcal{P}_{\gamma_{s, s_i}}^{m,n}, \| \cdot\|_{\mathcal{P}_{\gamma_{s, s_i}}^{m,n}} \right)_{n=1}^\infty$ is a global holomorphy type.
\end{corollary}
%\textcolor{blue}{Vejam que eu mudei muita coisa nesta se\c c\~ao. Fiz de acordo com o que eu estava entendendo, mas \'e bom vcs verificarem se n\~ao fiz nenhuma bobagem. Repito que vcs devem considerar a inclus\~ao de alguma demonstra\c c\~ao nesta se\c c\~ao.}

\section{Applications}\label{Sec5}
In this section we show that some well studied classes of multilinear operators can be recovered as particular cases of our abstract approach, and we also introduce new classes of multilinear operators that arise from our abstract point of view. Examples \ref{EEEE.2.11.} and \ref{Exemple4.6} assure that the main results of this paper apply to all classes listed in this section.  %  , we will presented some classic examples in the literature satisfy our abstract approach. In addition, we will presented new classes of multiple summing operators that satisfy our abstract approach, but not found in the literature.
%\begin{itemize}
\subsection{Multiple $(p, q_1,\dots, q_n)$-summing operators}

The class of multiple  $(p, q_1,\dots, q_n)$-summing operators, denoted by $\mathcal{L}_{ms(p, q_1,\dots, q_n)}(E_1,\dots, E_n; F)$, has been largely studied, see, e.g., \cite{19,20, 21}. This class is recovered by abstract framework by choosing
\begin{equation*}
\gamma_{s_k}(\cdot) = \ell_{q_k}^w(\cdot) \text{ and } 		\gamma_{s}\left(\cdot; \mathbb{N}^n\right) = \ell_p\left(\cdot; \mathbb{N}^n\right),
\end{equation*}
 for $1 \le q \le p \le \infty$ and $k=1,\dots, n$. %From the Examples \ref{EEEE.2.11.} and \ref{Exemple4.6}, we have that this class satisfies the main results of this paper.

\subsection{Multiple Cohen strongly $p$-summing operators}

The class of multiple Cohen strongly $p$-summing operators, denoted by $\mathcal{L}_{mCoh, p}(E_1,\dots, E_m; F)$, was studied in \cite{38}. In our abstract approach, it is recovered by choosing
\begin{equation*}
\gamma_{s_k}(\cdot) = \ell_p(\cdot) \text{ and } \gamma_s\left(\cdot; \mathbb{N}^n \right) = \ell_p\left<\cdot; \mathbb{N}^n \right>,
\end{equation*}
 where $1 < p < \infty$ and $k=1,\dots, n$. %From the Examples \ref{EEEE.2.11.} and \ref{Exemple4.6}, we have that this class satisfies the main results of this paper.

\subsection{Multiple mixing $(s, q, p)$-summing operators}

The concepts of mixed $m(s, q)$-summing sequences and multiple mixing $(s, q, p)$-summing operators were studied in \cite{BPSS15, 24, 33}. Just reminding the main definition, for %Let us present this concept and verify that this class is a particular case of our abstract approach.% Let us present this concept and show that the rule that assigns to each Banach space $E$ to the space of the $E$-valued mixed \textcolor{red}{ $m(s, q)$-summing} sequence is a finitely determined and linearly stable sequence class.
%\begin{definition}
%Let $E_1,\dots, E_n, F$ be Banach spaces,
$1 \leq p_1,\dots, p_n \le q \le s < \infty$, a continuous multilinear operator $T : E_1\times \cdots \times E_n \rightarrow F$ is multiple  $(s, q, p_1,\dots, p_n)$-mixing summing if
\begin{equation*}
\left(T\left(x_{j_1}^{(1)},\dots, x_{j_n}^{(n)} \right) \right)_{j_1,\dots, j_n=1}^{\infty} \in \ell_{m(s, q)}\left(F; \mathbb{N}^n\right),
\end{equation*}
whenever $\left(x_j^{(i)} \right)_{j=1}^{\infty} \in \ell_{p_i}^w(E_i)$, $i=1,\dots, n$.
%\end{definition}
%The space of all multiple   $(s, q, p_1,\dots, p_n)$-mixing summing is represented by $\prod_{mx(s, q, p)}$.
Note that, considering
\begin{equation*}
\gamma_s\left(\cdot; \mathbb{N}^n\right) = \ell_{m(s, q)}\left(\cdot; \mathbb{N}^n\right) \text{ and } \gamma_{s_k} = \ell_{p_i}^w,
\end{equation*}
for $i=1,\dots, n$, this class is a particular case of our general construction.%  of multipleand from the Examples \ref{EEEE.2.11.} and \ref{Exemple4.6}, we have that the class of multiple $(s, q, p)$-mixing summing operators satisfies the main results of this paper.

The next three subsections introduce new classes of multilinear operators which are particular cases of our  abstract framework, making clear that our results can also be applied to classes that had not been considered in the literature yet.

\subsection{Multiple strong $(s, q, p)$-mixing summing operators}

%We introduce a new class of multilinear operators:

%\begin{definition}
Let $1 \leq q \le s \le \infty$, $E_1,\dots, E_n, F$ be Banach spaces and $p \le q$. A continuous multilinear mapping $T : E_1\times \cdots \times E_n \rightarrow F$ is said to be multiple strong $(s, q, p)$-mixing summing if

\begin{equation*}
\left(T\left(x_{j_1}^{(1)},\dots, x_{j_n}^{(n)} \right) \right)_{j_1,\dots, j_n=1}^{\infty} \in \ell_p\left(F; \mathbb{N}^n\right)
\end{equation*}
wherever $\left(x_j^{(i)} \right)_{j=1}^{\infty} \in \ell_{m(s, q)}(E_i), i=1,\dots, n$.
%\end{definition}

 Choosing
\begin{equation*}
\gamma_s\left(\cdot; \mathbb{N}^n\right) = \ell_p\left(\cdot; \mathbb{N}^n\right) \text{ and } \gamma_{s_i} = \ell_{m(s, q)},
\end{equation*}
we  conclude that the class of all multiple $(s, q, p)$-mixing summing multilinear operators is a Banach multi-ideal for which the results of this paper apply. %  is represented by $\prod_{mx(s, q, p)}^{st}$. Note that, when we consider
%and from the Examples \ref{EEEE.2.11.} and \ref{Exemple4.6}, we have that the class of multiple strong $(s, q, p)$-mixing summing operators satisfies the main results of this paper.

\subsection{Multiple strong mid $p$-summing operators}

Let $1 \le p < \infty$, $n \in \mathbb{N}$ and $E_1,\dots, E_n, F$ be Banach spaces. A continuous multilinear operator $T : E_1 \times \cdots \times E_n \longrightarrow F$ is said to be multiple strong mid $p$-summing if
\begin{equation*}
\left(T\left(x_{j_1}^{(1)},\dots,  x_{j_n}^{(n)} \right) \right)_{j_1,\dots, j_n=1}^{\infty} \in \ell_p\left(F; \mathbb{N}^n\right)
\end{equation*}
whenever $\left(x_j^{(i)} \right)_{j=1}^{\infty} \in \ell_p^{mid}(E_i)$, $i=1,\dots, n$.

%Then the class of multiple strong mid $p$-summing applications will be denoted by $\mathcal{L}_{m, mid, p}^{st}(E_1,\dots, E_m; F)$.
Since the sequence classes $\ell_p$ and $\ell_p^{mid}$ are finitely determined and linearly stable (see \cite{BC17, 22}), the class of multiple strong mid $p$-summing multilinear operators is one more particular case of the classes studied in this paper.

%Since $\ell_{p}^{mid}\left(\cdot; \mathbb{N}^n \right)$ is compatible with $\ell_p\left(\cdot\right)$ (see Example \ref{EEEE.2.11.}). Then, this class satisfy the results in this paper.
%From the Examples \ref{EEEE.2.11.} and \ref{Exemple4.6}, we have that this class satisfies the main results of this paper.

\subsection{Multiple mid weakly $p$-summing operator}

Let $1 < p < \infty$, $n \in \mathbb{N}$ and $E_1,\dots, E_n, F$ be Banach spaces. A continuous multilinear application $T : E_1 \times \cdots \times E_n \longrightarrow F$ is said to be multiple weakly mid $p$-summing if
\begin{equation*}
\left(T\left(x_{j_1}^{(1)},\dots,  x_{j_n}^{(n)} \right) \right)_{j_1,\dots, j_n=1}^{\infty} \in \ell_p^{mid}\left(F; \mathbb{N}^n\right)
\end{equation*}
whenever $\left(x_j^{(i)} \right)_{j=1}^{\infty} \in \ell_p^{w}(E_i)$, $i=1,\dots, n$.

For the same reasons, the class of all multiple mid weakly $p$-summing multilinear operators is a particular instance of the classes studied in this paper as well. %  will be denoted by $\mathcal{L}_{m, mid, p}^{w}(E_1,\dots, E_m; F)$. In \cite{BC17, 22}, we see that the sequence classes $\ell_p^w(\cdot)$ and $\ell_p^{mid}(\cdot)$ are finitely determined and linearly stable.% Since $\ell_p^{mid}\left(\cdot; \mathbb{N}^n\right)$ is compatible with $\ell_p^{mid}(\cdot)$ (see Example \ref{EEEE.2.11.}). Then, this class satisfy the results in this paper.
%From the Examples \ref{EEEE.2.11.} and \ref{Exemple4.6}, we have that this class satisfies the main results of this paper.

%%%%

%It is important to note that the multilinear operator classes defined in the three subsections above have not been defined in the literature. But, we already have several characteristics of their multiple approach with the results presented in this work. Thus, this paper also presents results on classes of multiple operators that are not yet found in the literature.

%%%%%%%%%%%%%%%%%%%%%%%%%%%%%%%%%%%%%%%%%%%%%%%%%

\end{document}